\documentclass[12pt,reqno]{article}

\usepackage[usenames]{color}
\usepackage{amssymb}
\usepackage{graphicx}
\usepackage{amscd}

\usepackage[colorlinks=true,
linkcolor=webgreen,
filecolor=webbrown,
citecolor=webgreen]{hyperref}

\definecolor{webgreen}{rgb}{0,.5,0}
\definecolor{webbrown}{rgb}{.6,0,0}

\usepackage{color}
\usepackage{fullpage}
\usepackage{float}

\usepackage{graphics,amsmath,amssymb}
\usepackage{amsthm}
\usepackage{amsfonts}
\usepackage{latexsym}

\newcommand{\ul}{\underline}

\begin{document}

\vspace*{2.1cm}

\theoremstyle{plain}
\newtheorem{theorem}{Theorem}
\newtheorem{corollary}[theorem]{Corollary}
\newtheorem{lemma}[theorem]{Lemma}
\newtheorem{proposition}[theorem]{Proposition}
\newtheorem{obs}[theorem]{Observation}
\newtheorem{claim}[theorem]{Claim}

\theoremstyle{definition}
\newtheorem{definition}[theorem]{Definition}
\newtheorem{example}[theorem]{Example}
\newtheorem{remark}[theorem]{Remark}
\newtheorem{conjecture}[theorem]{Conjecture}
\newtheorem{question}[theorem]{Question}

\begin{center}

\vskip 1cm

{\Large\bf Total coloring of regular graphs of girth = degree + 1} %
\vskip 5mm
\large
Italo J. Dejter

University of Puerto Rico

Rio Piedras, PR 00936-8377

\href{mailto:italo.dejter@gmail.com}{\tt italo.dejter@gmail.com}
\end{center}


\begin{abstract} 
Let $2\le k\in\mathbb{Z}$. A total coloring of a
 $k$-regular simple graph via $k+1$ colors is an {\it efficient total coloring} if each color yields an efficient dominating set, where the efficient domination condition applies to the restriction of each color class to the vertex set. In this work, focus is set upon graphs of girth $k+1$.  
Efficient total colorings of finite connected simple cubic graphs of girth 4 are constructed starting at the 3-cube. 
It is conjectured that all of them are obtained by means of four basic operations. 
In contrast, the Robertson 19-vertex $(4,5)$-cage, the alternate union $Pet^k$ of a (Hamilton) $10k$-cycle with $k$ pentagon and $k$-pentagram $5$-cycles, for $k>1$ not divisible by 5, and its double cover $Dod^k$, contain TCs that are nonefficient. 
Applications to partitions into 3-paths and 3-stars are given.
\end{abstract}

\section{Introduction}\label{s0}

Given a simple connected graph $\Gamma$, a {\it total coloring} or {\it TC} of $\Gamma$ is a color assignment for the vertices and edges of $\Gamma$ such that no two incident or adjacent elements (vertices or edges)
are assigned the same color. A recent survey \cite{tc-as} contains an updated bibliography on TCs, with the TC Conjecture, posed independently by Behzad \cite{B1,B2} and by Vizing \cite{V}, that asserts that the total chromatic number of $\Gamma$ (namely, the least number of colors required by a TC of $\Gamma$) be either $\Delta(\Gamma)+1$ or $\Delta(\Gamma)+2$, where $\Delta$ is the largest degree of any vertex of $\Gamma$. 

 The TC Conjecture was established for cubic graphs \cite{Feng,Mazzu,Rosen,Vi}, meaning that the total chromatic number of cubic graphs is either 4 or 5. To decide whether a cubic graph $\Gamma$ has total chromatic number $\Delta(\Gamma)+1$, even for bipartite cubic graphs, is NP-hard \cite{Arroyo}.

In the present work, we study the relations between total colorings of $k$-regular graphs of girth $k+1$ ($1<k\in\mathbb{Z}$) with the presence of efficient dominating sets  \cite{worst,Tomai,D73,Deng, EDS,Knor} in those graphs. For such purpose, we introduce the following definition.

\begin{definition}\label{ahora} A coloring of a connected $k$-regular simple graph $\Gamma$ ($2\le k\in\mathbb{Z}$) is said to be an {\it efficient TC}, or {\it ETC}, (and $\Gamma$ said to be {\it ETCed}), if:
\begin{enumerate}
\item[\bf(a)] (TC condition) each $v\in V(\Gamma)$ together with its neighbors are assigned all the colors in $[k+1]=\{0,1,\ldots,k\}$ via a bijection $N[v]=N(v)\cup\{v\}\leftrightarrow[k+1]$, where $N[v]$ and $N(v)$ are the {\it closed neighborhood} of $v$ and the {\it open neighborhood} of $v$, respectively \cite{D73}; 
\item[\bf(b)] the TC in item (a) partitions $V(\Gamma)$ into $k+1$ {\it efficient dominating sets} ({\it EDS}), also called {\it perfect codes}, namely independent (stable) subsets $S_i\subseteq V(\Gamma)$ such that for any $v\in V(\Gamma)\setminus S_i$, $|N[v]\cap S_i|=1$,
where $i$ varies in $[k+1]$, \cite{worst,Tomai,D73,Deng,EDS,Knor}.
\end{enumerate}\end{definition}

\noindent Under conditions (a)-(b), it is seen that the total chromatic number of $\Gamma$ is $\Delta(\Gamma)+1$. 

\begin{remark}\label{3j}
For $1<j\in\mathbb{Z}$, consider the $3j$-cycle graph $C_{3j}=(v_1,e_1,v_2,e_2,\ldots,e_{3j-1},v_{3j},$ $e_{3j})$, where $e_i$ is the edge with end-vertices $v_i$ and $v_{i+1}$, for $1\le i\le 3j-1$, and $e_{3j}$ is the edge with end-vertices $v_{3j}$ and $v_1$. Let the vertices and edges of $C_{3j}$ be respectively colored $(0,1,2,0,1,2,\ldots,0,1,2)$. This clearly yields an ETC of $C_{3j}$. 
\end{remark}

In the rest of this work, finite connected simple $k$-regular graphs $\Gamma$ ($k\ge 2$) of girth $k+1$ 
are dealt with. Our purpose is to determine ETCs of such graphs. Some of these ETCs yield {\it edge-girth colorings} (see Definition~\ref{egc}, below) on the prism $\Gamma\square K_2$, where $K_2=P_2$ is the complete graph on two vertices, that is the path $P_k$ consisting of $k=2$ vertices. 
  
 \begin{definition}\label{egc} Let $\Gamma$ be a finite connected $k$-regular simple graph of girth $k+1$.
An {\it edge-girth coloring}, or {\it EGC}, of $\Gamma$ is a proper edge coloring via $k+1$ colors, each girth cycle colored with $k+1$ colors, each color used precisely once. 
\end{definition}

In order to present our results, we need every girth cycle $C$ of $\Gamma$ to be colored with $k$ colors, each color used exactly once on the vertices of $C$ and exactly once on the edges of $C$. This leads  to the following Definition~\ref{hoy}, in which, in addition, the total coloring of the girth cycles is combined with the concept of ETC in Definition~\ref{ahora}. The final Subsection~\ref{7} shows that such a combination of total coloring with efficient domination is not anymore the case for $k=4$, see Remark~\ref{alu}. 

\begin{definition}\label{hoy} A {\it vertex-edge-girth coloring}, or VEGC, of a connected simple graph $\Gamma$ of girth $k$ is a TC of $\Gamma$ in which each girth cycle is colored with $k$ colors, each color used just once on vertices and also just once on edges. In addition, an ETC of $\Gamma$ that is also VEGC will be said to be an {\it efficient total girth coloring}, or {\it ETGC}, of $\Gamma$.
\end{definition}

The results of Section~\ref{s1} involve the existence of ETGCs in finite connected simple $k$-regular graphs $\Gamma$ with girth $k +1$. They also involve the existence  of cubic graphs $\Gamma$ of corresponding EGCs on the prisms $\Gamma\square K_2$, for which we need the following Definition~\ref{ortho}.

\begin{definition}\label{ortho}
Given a connected simple graph $\Gamma$ and a TC $C$ of $\Gamma$, if $C$ can be extended to two ETCs $C'$ and $C''$ differing in color on every edge of $\Gamma$, then $C'$ and $C''$ are said to be {\it orthogonal} ETCs and their induced edge colorings are also said to be {\it orthogonal}.
\end{definition}

Initially, two cases of finite connected simple cubic graphs $\Gamma$ of girth 4 are considered, namely planar and toroidal graphs, in Subsections~\ref{planar} 
 and~\ref{toroidal}, respectively. 
 Subsection~\ref{alg} deals with algorithmic aspects of the construction of ETCs. 
 Subsection~\ref{hand} extends the results that far obtained to surfaces of larger genera than those of the sphere and the torus, 
 leaving as Conjecture~\ref{con1} that all existing ETGCs on finite connected simple cubic graphs of girth 4 are obtained from the 3-cube (as in Theorem~\ref{t1}) by applying four constructive operations, namely: {\it periodic extensions} (see Definitions~\ref{pe} and~\ref{pets}), {\it accordion unfoldings} (see Definitions~\ref{unfold} and~\ref{reyes}), {\it cycle exchanges} (see Remark~\ref{remark}) and {\it ETCing} (see Definition~\ref{sabado}). To define properly such operations, planar and toroidal graphs $\Gamma$ are presented via {\it cutouts} and {\it normal cutouts} respectively in Definitions~\ref{cutout} and ~\ref{normal}. To recover $\Gamma$ from such cutouts, identifications of the resulting opposite sides in such cutouts must be performed. The faces of such cutouts (seen as plane graphs) are delimited by cycles of certain lengths $\ell$ that we denominate $\ell$-{\it belts} in Definition~\ref{belt}.
 
\begin{definition}\label{belt}
An $\ell$-{\it belt} in a plane graph $\Gamma$ is a cycle bounding a face of $\Gamma$ of length $\ell$.
\end{definition}
 
Definition~\ref{belt} is to be used in the four mentioned constructive operations in the developments of 
 Theorems~\ref{t1} and~\ref{starting} of Subsection~\ref{planar} on planar graphs, Theorems~\ref{16vertex} and~\ref{sonidos} of Subsection~\ref{toroidal} on toroidal graphs, Theorem~\ref{tt} of Subsection~\ref{alg}, that uses algorithmic aspects of ETGCs, and Theorem~{\ref{tu} of Subsection~\ref{hand}, on raising the genus of $\Gamma$ via handles.

\begin{remark}\label{alu} In contrast to the case of cubic graphs treated in Subsections~\ref{planar}--\ref{chance}, Subsection~\ref{7} shows that the unique $[4,5]$-cage, namely the 19-vertex Robertson graph~\cite{Rob}, contains TCs that are nonefficient, as do both the union $Pet^k$ of a (Hamilton) $10k$-cycle with $k$ pentagon and $k$-pentagram $5$-cycles disposed alternatively ($k>1$ not divisible by 5,
 also obtained by joining $k$ copies of the Petersen graph $Pet$ via a 1-factor, though actually $Pet^k$                                                                                                                                                                                                                                            contains $2k$ copies of $Pet$), 
  and their double covers $Dod^k$ based on $k$ copies of the dodecahedral graph $Dod$. 
  \end{remark}

\section{Regular graphs of girth = degree + 1}\label{s1} 

\begin{question}\label{q1} Let $\Gamma$ be a graph of regular degree $k\ge 3$ and girth $k+1$.  Does $\Gamma$ possess an ETGC with $k+1$ colors?\end{question} 

\begin{theorem}\label{fo}
Let $\Gamma$ be a finite connected plane simple cubic graph of girth 4. If $\Gamma$ has an ETC with four colors, then $|V(\Gamma)|\equiv 0 \mod 4$ and $\Gamma$ has only $\ell$-belts with $\ell\equiv 0 \mod 4$. 
\end{theorem}

\begin{proof}
An $\ell$-belt with $\ell\not\equiv 0 \mod 4$ necessarily contains two vertices of a common color at a distance less than 3.
Therefore, $\Gamma$ cannot have an ETC with four colors. 
\end{proof}

\subsection{Planar cases}\label{planar}

\begin{definition}\label{cutout}
A {\it cutout} of a connected simple planar graph $\Gamma$ is a closed rectangle $\Phi=[0,x]\times[0,y]$ of $\mathbb{R}^2$ with $V(\Gamma)\subset\Phi\cap\mathbb{Q}^2$ and $E(\Gamma)$ given by vertical and horizontal segments of unit or unit-fraction lengths such that the identification of segments 
$\{0\}\times[0,y]\equiv\{x\}\times[0,y]$, 
(or $[0,x]\times\{0\}\equiv[0,x]\times\{y\}$),
yields a representation of $\Gamma$ with 
$(0,y)\equiv(x,y)$, (or $(0,0)\equiv(0,y)$), respectively $(0,0)\equiv(x,0)$, (or $(x,0)\equiv(x,y)$), 
representing a sole top, (or left), respectively bottom, (or right), vertex of $\Gamma$. 
\end{definition}

The identification in Definition~\ref{cutout} yields a  plane representation of $\Gamma$ with belts delimiting faces equivalent to those shown in the cutout itself or their continuation after a  periodic extension as in Definition~\ref{pe}.

 \begin{definition}\label{pe} A {\it periodic extension} of a cutout $\Phi$ of a connected simple planar graph $\Gamma$ is a cutout formed by the union of two or more copies $\Phi_1,\Phi_2,\ldots,\Phi_n$ of $\Phi$ ($1\le n\in\mathbb{Z}$) along the common linear vertical (or horizontal) borders between each $\Phi_i$ and $\Phi_{i+1}$ ($1\le i<n$), that is by identifying the right (or top) border of $\Phi_i$ and the left (or bottom) border of $\Phi_{i+1}$.
 \end{definition}

\begin{theorem}\label{t1}
The binary 3-cube graph $Q_3$ and the prisms $C_{4j}\square K_2$ (where $1<j\in\mathbb{Z}$) have ETGCs via color set $[4]$ in two mutually orthogonal ways. Moreover, any two resulting orthogonal ETGCs guarantee an EGCs in the corresponding 4-cube graph $Q_4=Q_3\square K_2$ or prism $(C_{4j}\square K_2)\square K_2$.
\end{theorem}

\begin{proof}
The 3-cube $Q_3$ can be represented as an ETGC in two different ways via the two shown cutouts in display (\ref{oct}), with $Q_3$ recoverable by gluing in parallel the leftmost and rightmost vertical segments in the cutouts; these are separated by the inequality symbol $\neq$.
 \begin{eqnarray}\label{oct}\begin{array}{lll}
0\hspace*{2.2mm} _-^3\hspace*{2.2mm} 1\hspace*{2.2mm} _-^0\hspace*{2.2mm} 2\hspace*{2.2mm} _-^1\hspace*{2.2mm} 3\hspace*{2.2mm} _-^2\hspace*{2.2mm} 0&&
0\hspace*{2.2mm} _-^2\hspace*{2.2mm} 1\hspace*{2.2mm} _-^3\hspace*{2.2mm} 2\hspace*{2.2mm} _-^0\hspace*{2.2mm} 3\hspace*{2.2mm} _-^1\hspace*{2.2mm} 0\\

\!_1|\hspace*{6mm}_2|\hspace*{6mm}_3|\hspace*{6mm}_0|\hspace*{6mm}_1|&\neq&
\!_3|\hspace*{6mm}_0|\hspace*{6mm}_1|\hspace*{6mm}_2|\hspace*{6mm}_3|\\

2\hspace*{2.2mm} _0^-\hspace*{2.2mm}3\hspace*{2.2mm} _1^-\hspace*{2.2mm} 0\hspace*{2.2mm} _2^-\hspace*{2.2mm} 1\hspace*{2.2mm} _3^-\hspace*{2.2mm} 2&&
2\hspace*{2.2mm} _1^-\hspace*{2.2mm}3\hspace*{2.2mm} _2^-\hspace*{2.2mm} 0\hspace*{2.2mm} _3^-\hspace*{2.2mm} 1\hspace*{2.2mm} _0^-\hspace*{2.2mm} 2 \\
\end{array}\end{eqnarray}
The vertices and edges in either cutout are given via large and small (color) numbers, respectively, with the small numbers accompanying horizontal and vertical segments indicating the edges. The vertex colors are similar in both cutouts but the edge colors differ, so these edge colorings are orthogonal. Antipodal vertices of $Q_3$, differing 3 in distance, receive the same color number in either TC, so they are (orthogonal) ETCs. Moreover, by adding an edge between each pair of antipodal vertices $(v_0,v_1,v_2),(\ul{v_0},\ul{v_1},\ul{v_2})$, (with $v_i\in\{0,1\}$ and $\ul{v_i}=$ binary complement of $v_i$, for $i=0,1,2$), a copy of the complete bipartite graph $K_{4,4}$ is obtained. By assigning the common color of the vertices of each such pair to the corresponding added edge, a proper edge coloring of $K_{4,4}$ is obtained, not an EGC because of the presence of 4-cycles using just two colors on alternate edges. 
The 4-cube $Q_4$ acquires an EGC as a lifting of such proper edge coloring of $K_{4,4}$ via the canonical projection $Q_4\rightarrow K_{4,4}$, as illustrated for example in the center and right of Fig. 4 in \cite{DDD}, with all the 2-color 4-cycles of $K_{4,4}$ lifting to corresponding 2-color 8-cycles in $Q_4$. Another way to visualize such EGC on $Q_4$ is to consider two parallel concentric copies of $Q_3$, one larger than the other, with the vertices $(v_0,v_1,v_2)$ of $Q_3$ represented as $(v_0,v_1,v_2,0)$ in the inner 3-cube $Q_3'$ and as $(\ul{v_0},\ul{v_1},\ul{v_2},1)$ in the outer 3-cube $Q_3''$. The claimed EGC in $Q_4$ is obtained by lifting the edge colors of $Q_3$ as edge colors of $Q_4$ for the edges with a common last coordinate in $\{0,1\}$ and by assigning the common color of its end-vertices to each other edge, which differ in their last coordinate, forming a 1-factor between $Q_3'$ and $Q_3''$. 
 
Now if we glue successively a finite number of copies of say the leftmost cutout in display (\ref{oct}) and identify in parallel the first and last edges of the resulting graph, 
that is apply Definition~\ref{pe} of periodic extension,
a prism $C_{4j}\square K_2$ is obtained with an ETGC, by iterated continuation of the numerical color pattern. This is exemplified in display (\ref{oct2}) by concatenation of two colored cutouts as in the left of display (\ref{oct}) to produce an ETGC in the polygonal prism $C_8\square K_2$ to the right of the periodic-extension indication $(^{\times n}_\rightarrow)$.
Iteration of the continuation in display (\ref{oct2}) indicated by the operation $(_\rightarrow^{\times n})$ yields the statement of the theorem.

\begin{eqnarray}\label{oct2}\begin{array}{lll}
0\hspace*{2.2mm} _-^3\hspace*{2.2mm} 1\hspace*{2.2mm} _-^0\hspace*{2.2mm} 2\hspace*{2.2mm} _-^1\hspace*{2.2mm} 3\hspace*{2.2mm} _-^2\hspace*{2.2mm} 0 &&
0\hspace*{2.2mm} _-^3\hspace*{2.2mm} 1\hspace*{2.2mm} _-^0\hspace*{2.2mm} 2\hspace*{2.2mm} _-^1\hspace*{2.2mm} 3\hspace*{2.2mm} _-^2\hspace*{2.2mm}  0\hspace*{2.2mm} _-^3\hspace*{2.2mm} 1\hspace*{2.2mm} _-^0\hspace*{2.2mm} 2\hspace*{2.2mm} _-^1\hspace*{2.2mm} 3\hspace*{2.2mm} _-^2\hspace*{2.2mm} 0 \\
\!_1|\hspace*{6mm}_2|\hspace*{6mm}_3|\hspace*{6mm}_0|\hspace*{6mm}_1|&(_\rightarrow^{\times n})&
\!_1|\hspace*{6mm}_2|\hspace*{6mm}_3|\hspace*{6mm}_0|\hspace*{6mm}_1|\hspace*{6mm}_2|\hspace*{6mm}_3|\hspace*{6mm}_0|\hspace*{6mm}_1|\\
2\hspace*{2.2mm} _0^-\hspace*{2.2mm}3\hspace*{2.2mm} _1^-\hspace*{2.2mm} 0\hspace*{2.2mm} _2^-\hspace*{2.2mm} 1\hspace*{2.2mm} _3^-\hspace*{2.2mm} 2 &&
2\hspace*{2.2mm} ^-_0\hspace*{2.2mm}3\hspace*{2.2mm} ^-_1\hspace*{2.2mm} 0\hspace*{2.2mm} ^-_2\hspace*{2.2mm} 1\hspace*{2.2mm} ^-_3\hspace*{2.2mm} 2\hspace*{2.2mm} ^-_0\hspace*{2.2mm}3\hspace*{2.2mm} ^-_1\hspace*{2.2mm} 0\hspace*{2.2mm} ^-_2\hspace*{2.2mm} 1\hspace*{2.2mm} ^-_3\hspace*{2.2mm} 2 \\
\end{array}\end{eqnarray}
\end{proof} 

\begin{remark} In the case of the proof of Theorem~\ref{t1} for $Q_3$, the partition in Definition~\ref{ahora}(b) of ETGC in Section~\ref{s0} is composed by the colored subsets:
\begin{eqnarray}\label{eqn0}S_0=\{000,111\},\; S_1=\{100,011\},\; S_2=\{010,101\},\; S_3=\{001,110\},\end{eqnarray} that is the pairs of opposite, or antipodal, or complementary, vertices in color numbers 0,1,2,3, respectively. In addition, a 1-factor $F_i$ exists in each of the 6-cycles of $Q_3$ that form the complements $Q_3\setminus S_i$, ($i=0,1,2,3$), of the pairs $S_i$ in $Q_3$, with the color $i$ assignment adopted for the vertices of $S_i$ extended to the  edges of $F_i$. This takes care of all the edges of $Q_3$, as follows:   

\begin{eqnarray}\label{cuatro}\begin{array}{c}
F_0=\{(100,101),(010,110),(001,011)\}\subset E(Q_3)\setminus S_0\\
F_1=\{(000,010),(001,101),(110,111)\}\subset E(Q_3)\setminus S_1\\
F_2=\{(000,001),(100,110),(011,111)\}\subset E(Q_3)\setminus S_2\\
F_3=\{(000,100),(010,011),(101,111)\}\subset E(Q_3)\setminus S_3
\end{array}\end{eqnarray}

With respect to iterated continuation of the numerical color pattern in (\ref{oct}) for a prism $C_{4j}\square K_2$, where $1<j\in\mathbb{Z}$, the color cycle accompanying a cycle composed by those edges whose vertices are not colored 0 (forming an EDS $S_0$)
is of the form $$(020103)^j=(020103020103\cdots 020103),$$ where 020103 is concatenated $j$ times before closing a cycle. Then, the cycle $(C_{4j}\square K_2)\setminus S_0$ contains a 1-factor $F_0$ whose vertices are colored 0.
\end{remark}

\begin{remark} It is proved in \cite{Amanda} that the generalized Petersen graph $G(n,1)=C_n\square K_2$ for $n\ne 5$ has total chromatic number $\Delta +1$.  Such $G(n,1)$ has a TC, but it is not necessarily ETC or ETGC, as can be seen in the cutouts for $n=6,7$ in display (\ref{chet}).

\begin{eqnarray}\label{chet}\begin{array}{cccc}
0\hspace*{2.2mm} _-^1\hspace*{2.2mm} 2\hspace*{2.2mm} _-^0\hspace*{2.2mm} 1\hspace*{2.2mm} _-^2\hspace*{2.2mm} 0\hspace*{2.2mm} _-^1\hspace*{2.2mm} 2
\hspace*{2.2mm} _-^0\hspace*{2.2mm} 1\hspace*{2.2mm} _-^2\hspace*{2.2mm} 0
&&&
3\hspace*{2.2mm} _-^0\hspace*{2.2mm} 1\hspace*{2.2mm} _-^3\hspace*{2.2mm} 2\hspace*{2.2mm} _-^1\hspace*{2.2mm} 0\hspace*{2.2mm} _-^2\hspace*{2.2mm} 1
\hspace*{2.2mm}_-^0\hspace*{2.2mm} 2\hspace*{2.2mm} _-^1\hspace*{2.2mm} 0\hspace*{2.2mm} _-^2\hspace*{2.2mm} 3\\

\!_3|\hspace*{6mm}_3|\hspace*{6mm}_3|\hspace*{6mm}_3|\hspace*{6mm}_3|\hspace*{6mm}_3|\hspace*{6mm}_3|&&&
\!_1|\hspace*{6mm}_2|\hspace*{6mm}_0|\hspace*{6mm}_3|\hspace*{6mm}_3|\hspace*{6mm}_3|\hspace*{6mm}_3|\hspace*{6mm}_1|\\

1\hspace*{2.2mm} _2^-\hspace*{2.2mm}0\hspace*{2.2mm} _1^-\hspace*{2.2mm} 2\hspace*{2.2mm} _0^-\hspace*{2.2mm} 1\hspace*{2.2mm} _2^-\hspace*{2.2mm} 0\hspace*{2.2mm} _1^-\hspace*{2.2mm} 2\hspace*{2.2mm} _0^-\hspace*{2.2mm} 1
&&&
2\hspace*{2.2mm} _3^-\hspace*{2.2mm}0\hspace*{2.2mm} _1^-\hspace*{2.2mm} 3\hspace*{2.2mm} _2^-\hspace*{2.2mm} 1\hspace*{2.2mm} _0^-\hspace*{2.2mm} 2
\hspace*{2.2mm} _1^-\hspace*{2.2mm} 0\hspace*{2.2mm} _2^-\hspace*{2.2mm} 1\hspace*{2.2mm} _0^-\hspace*{2.2mm} 2 
\\
\end{array}\end{eqnarray}
Of the seven girth 4-cycles of $P(7,1)$, only the three on the left of its cutout (on the right of (\ref{chet})) have their vertex sets and edge sets in bijective correspondence with the color set $[4]=\{0,1,2,,3\}$. (Compare for example with Theorem~\ref{Gamma} below, where just 14 of the 54 girth 5-cycles have their vertex set and edge set in bijective correspondence with $[5]$). 
\end{remark}

The periodic extensions of cutouts of Definition~\ref{pe} generalize as in the following Definition~\ref{unfold}, to be useful in Theorem~\ref{starting} and Examples~\ref{ex1}--\ref{ex3}.

\begin{definition}\label{unfold}
Given a cutout $\Phi$ of a connected simple planar cubic graph $\Gamma$ of girth 4, an ({\it accordion}) {\it unfolding} of $\Phi$ is a cutout $\Phi'$ of a connected planar cubic graph $\Gamma'$ obtained by the replacement of 4-belts of the form $P_2\square P_2$ by copies of $P_2\square P_{2\ell}$, where $1<\ell\in\mathbb{Z}$.
\end{definition}

The unfoldings that are periodic extensions as in Definition~\ref{pe} are obtained by operating upon the belts of  a cutout of $Q_3$. Once we started with the belts of a cutout of.$Q_3$, much more can be obtained via successive unfoldings of the belts of intermediate graphs obtained in the process, as can be seen in the next Theorem~\ref{starting} and subsequent Examples~\ref{ex1}--\ref{ex3}.

\begin{theorem}\label{starting} Starting with the 4-belts of a cutout of $Q_3$,
successive unfoldings lead to an infinite family of planar graphs $\Gamma'$. Each such $\Gamma'$ lacking $\ell$-belts with $\ell\not\equiv 0 \mod 4$ in the corresponding cutout has ETGCs. Moreover, there is a VEGC in the corresponding prism.
\end{theorem}

\begin{proof} Let $\Gamma$ be an ETCed graph as in the statement of Theorem~\ref{t1}. Let $X=(AcBdCaDb)$ be a colored 4-belt of $\Gamma$ as on the left of displays (\ref{octubre}) and (\ref{noviembre}). Consider the graph obtained as the union of the 4-belts
$X$, $Y=(DaCbAdBc)$ and $Z=(BdAcDbCa)$ with successive intersections, (colored copies of $K_2=P_2$), $X\cap Y=DaC$ and $Y\cap Z=BdA$. Then, the graph $\Gamma'$ obtained by replacing the sides $AbD$ and $BdC$ of $\Gamma$ by the paths $AbDcBaC$ and $BdCbAcD$, respectively, with the addition of edges to include $X$, $Y$ and $Z$, namely edges $DbA$, $BdA$ and $CbD$, is in the infinite family mentioned in the statement. 
This transformation $P_2\square P_2\rightarrow P_2\square P_4$ is illustrated in displays (\ref{octubre}) and (\ref{noviembre}), further extended to respective transformations $P_2\square P_2\rightarrow P_2\square P_6$ and $P_2\square P_2\rightarrow P_2\square P_8$, via composition of extensions. 

\begin{eqnarray}\label{octubre}\begin{array}{lllll}
A\hspace*{2.2mm} _-^b\hspace*{2.2mm} D&&A\hspace*{2.2mm} _-^b\hspace*{2.2mm} D\hspace*{2.2mm} _-^c\hspace*{2.2mm} B\hspace*{2.2mm} _-^a\hspace*{2.2mm} C\hspace*{2.2mm}&&A\hspace*{2.2mm} _-^b\hspace*{2.2mm} D\hspace*{2.2mm} _-^c\hspace*{2.2mm} B\hspace*{2.2mm} _-^a\hspace*{2.2mm} C\hspace*{2.2mm} _-^d\hspace*{2.2mm}  A\hspace*{2.2mm} _-^b\hspace*{2.2mm} D\\
\!_c|\hspace*{7.5mm}_a|&\rightarrow&\!_c|\hspace*{7.5mm}_a|\hspace*{7.5mm}_d|\hspace*{7.5mm}_b|&\rightarrow&\!_c|\hspace*{7.5mm}_a|\hspace*{7.5mm}_d|\hspace*{7.5mm}_b|\hspace*{7.5mm}_c|\hspace*{7.5mm}_a|\\
B\hspace*{2.2mm} ^-_d\hspace*{2.2mm}C&&B\hspace*{2.2mm} ^-_d\hspace*{2.2mm}C\hspace*{2.2mm} ^-_b\hspace*{2.2mm} A\hspace*{2.2mm} ^-_c\hspace*{2.2mm} D&&B\hspace*{2.2mm} ^-_d\hspace*{2.2mm}C\hspace*{2.2mm} ^-_b\hspace*{2.2mm} A\hspace*{2.2mm} ^-_c\hspace*{2.2mm} D\hspace*{2.2mm} ^-_a\hspace*{2.2mm} B\hspace*{2.2mm} ^-_d\hspace*{2.2mm}C\\
\end{array}\end{eqnarray}

\begin{eqnarray}\label{noviembre}\begin{array}{lllll}
A\hspace*{2.2mm} _-^b\hspace*{2.2mm} D&&A\hspace*{2.2mm} _-^b\hspace*{2.2mm} D\hspace*{2.2mm} _-^c\hspace*{2.2mm} B\hspace*{2.2mm} _-^a\hspace*{2.2mm} C\hspace*{2.2mm}&&A\hspace*{2.2mm} _-^b\hspace*{2.2mm} D\hspace*{2.2mm} _-^c\hspace*{2.2mm} B\hspace*{2.2mm} _-^a\hspace*{2.2mm} C\hspace*{2.2mm} _-^d\hspace*{2.2mm}  A\hspace*{2.2mm} _-^b\hspace*{2.2mm} D\hspace*{2.2mm} _-^c\hspace*{2.2mm} B\hspace*{2.2mm} _-^a\hspace*{2.2mm} C\hspace*{2.2mm}\\
\!_c|\hspace*{7.5mm}_a|&\rightarrow&\!_c|\hspace*{7.5mm}_a|\hspace*{7.5mm}_d|\hspace*{7.5mm}_b|&\rightarrow&\!_c|\hspace*{7.5mm}_a|\hspace*{7.5mm}_d|\hspace*{7.5mm}_b|\hspace*{7.5mm}_c|\hspace*{7.5mm}_a|\hspace*{7.5mm}_d|\hspace*{7.5mm}_b|\\
B\hspace*{2.2mm} ^-_d\hspace*{2.2mm}C&&B\hspace*{2.2mm} ^-_d\hspace*{2.2mm}C\hspace*{2.2mm} ^-_b\hspace*{2.2mm} A\hspace*{2.2mm} ^-_c\hspace*{2.2mm} D&&B\hspace*{2.2mm} ^-_d\hspace*{2.2mm}C\hspace*{2.2mm} ^-_b\hspace*{2.2mm} A\hspace*{2.2mm} ^-_c\hspace*{2.2mm} D\hspace*{2.2mm} ^-_a\hspace*{2.2mm} B\hspace*{2.2mm} ^-_d\hspace*{2.2mm}C\hspace*{2.2mm} ^-_b\hspace*{2.2mm} A\hspace*{2.2mm} ^-_c\hspace*{2.2mm} D\hspace*{2.2mm}\\
\end{array}\end{eqnarray}
Iteration of such extensions leads to all graphs $\Gamma'$ in the cited family. Only those graphs $\Gamma'$ lacking any $\ell$-belts with $\ell\not\equiv 0 \mod 4$  have ETGC.

All graphs $\Gamma'$ obtained via the modifications presented above are planar finite cubic graphs of girth 4.
The ETGCs obtained are the only possible TCs in planar cubic graphs of girth 4. The two orthogonal ETGCs of Theorem~\ref{t1} give place clearly to a VEGC in the corresponding prism. 
\end{proof}   

\begin{example}\label{ex1}

 We obtain the graphs with cutouts represented in the upper section of display (\ref{oct6}) with ETGCs as indicated, 
 resulting from two and four modifications as on the left sides of (\ref{octubre}) and~(\ref{noviembre}), respectively,
performed on the left and right sides of display (\ref{oct2}). Additionally, let us replace on the cutout $\Phi$ to the right of the indication $(^{\times n}_{\rightarrow})$ the two edges that are indicated $0\hspace*{2.2mm}^{..}_2\hspace*{2.2mm}1$ (via diaereses, or transpose colons, instead of $0\hspace*{2.2mm}^-_2\hspace*{2.2mm}1$) with two edges not in $\Phi$ but forming a 4-cycle with the four participating vertices, suggested to the right of $\Phi$ as an auxiliary square $K_2\square K_2$ (as in display (\ref{octaedro}) and Remark~\ref{remark}). This replacement yields a 
toroidal graph with an ETGC that was not produced in the context of Theorem~\ref{starting}, since the two new edges are not in $\Phi$, so a handle has to be added to $\Phi$, making the original planar graph, say $\Gamma$, into a toroidal graph $\Gamma'$, where $\Phi$ loses the two edges  $0\hspace*{2.2mm}^{..}_2\hspace*{2.2mm}1$, sensed by the presence of two resulting 6-cycles in $\Phi$ that can be taken as the two circular borders of the handle that carry the two new non-$\Phi$ edges.  

\begin{eqnarray}\label{oct6}\begin{array}{cccc|c}
0\hspace*{2.2mm} _-^3\hspace*{2.2mm} 1\hspace*{2.2mm} _-^0\hspace*{2.2mm} 2\hspace*{2.2mm} _-^1\hspace*{2.2mm} 3\hspace*{2.2mm} _-^2\hspace*{2.2mm} 0&&
0\hspace*{2.2mm} _-^3\hspace*{2.2mm} 1\hspace*{2.2mm} _-^0\hspace*{2.2mm} 2\hspace*{2.2mm} _-^1\hspace*{2.2mm} 3\hspace*{2.2mm} _-^2\hspace*{2.2mm}  0\hspace*{2.2mm} _-^3\hspace*{2.2mm} 1\hspace*{2.2mm} _-^0\hspace*{2.2mm} 2\hspace*{2.2mm} _-^1\hspace*{2.2mm} 3\hspace*{2.2mm} _-^2\hspace*{2.2mm} 0&& \\
\!_1|\hspace*{6mm}_2|\hspace*{6mm}_3|\hspace*{6mm}_0|\hspace*{6mm}_1|&&
\!_1|\hspace*{6mm}_2|\hspace*{6mm}_3|\hspace*{6mm}_0|\hspace*{6mm}_1|\hspace*{6mm}_2|\hspace*{6mm}_3|\hspace*{6mm}_0|\hspace*{6mm}_1|&&\\
2\hspace*{2.2mm} ^-_0\hspace*{2.2mm}3\hspace*{7mm} 0\hspace*{2.2mm} ^-_2\hspace*{2.2mm} 1\hspace*{7mm} 2&&
\,2\hspace*{2.2mm} ^-_0\hspace*{2.2mm}3\hspace*{7mm} 0\hspace*{2.2mm} ^{..}_2\hspace*{2.6mm} 1\hspace*{7mm} 2\hspace*{2.2mm} ^-_0\hspace*{2.2mm}3\hspace*{7mm} 0\hspace*{2.2mm} ^{..}_2\hspace*{2.6mm} 1\hspace*{7mm}2 &&0\hspace*{2.2mm}^{..}_2\hspace*{2.2mm}1\\
\!_3|\hspace*{6mm}_1|\hspace*{6mm}_1|\hspace*{6mm}_3|\hspace*{6mm}_3|&(^{\times n}_\rightarrow)&
\!_3|\hspace*{6mm}_1|\hspace*{6mm}_1|\hspace*{6mm}_3|\hspace*{6mm}_3|\hspace*{6mm}_1|\hspace*{6mm}_1|\hspace*{6mm}_3|\hspace*{6mm}_3|&&\,|_2\hspace{6mm}|_2\\
1\hspace*{2.2mm} ^-_2\hspace*{2.2mm}0\hspace*{7mm} 3\hspace*{2.2mm} ^-_0\hspace*{2.2mm} 2\hspace*{7mm} 1&&
1\hspace*{2.2mm} ^-_2\hspace*{2.2mm}0\hspace*{7mm} 3\hspace*{2.2mm} ^-_0\hspace*{2.2mm} 2\hspace*{7mm} 1\hspace*{2.2mm} ^-_2\hspace*{2.2mm}0\hspace*{7mm} 3\hspace*{2.2mm} ^-_0\hspace*{2.2mm} 2\hspace*{7mm}1 &&1\hspace*{2.2mm}^{..}_2\hspace*{2.2mm}0\\
\!_0|\hspace*{6mm}_3|\hspace*{6mm}_2|\hspace*{6mm}_1|\hspace*{6mm}_0|&&
\!_0|\hspace*{6mm}_3|\hspace*{6mm}_2|\hspace*{6mm}_1|\hspace*{6mm}_0|\hspace*{6mm}_3|\hspace*{6mm}_2|\hspace*{6mm}_1|\hspace*{6mm}_0|&&\\
3\hspace*{2.2mm} ^-_1\hspace*{2.2mm}2\hspace*{2.2mm} ^-_0\hspace*{2.2mm} 1\hspace*{2.2mm} ^-_3\hspace*{2.2mm} 0\hspace*{2.2mm} ^-_2\hspace*{2.2mm} 3 &&
3\hspace*{2.2mm} ^-_1\hspace*{2.2mm}2\hspace*{2.2mm} ^-_0\hspace*{2.2mm} 1\hspace*{2.2mm} ^-_3\hspace*{2.2mm} 0\hspace*{2.2mm} ^-_2\hspace*{2.2mm} 3\hspace*{2.2mm} ^-_1\hspace*{2.2mm}2\hspace*{2.2mm} ^-_0\hspace*{2.2mm} 1\hspace*{2.2mm} ^-_3\hspace*{2.2mm} 0\hspace*{2.2mm} ^-_2\hspace*{2.2mm} 3 &&\\
=&&\neq&&\\
0\hspace*{2.2mm} ^-_3\hspace*{2.2mm} 1\hspace*{2.2mm} ^-_0\hspace*{2.2mm}2\hspace*{2.2mm} ^-_1\hspace*{2.2mm}3\hspace*{2.2mm}^-_2\hspace*{2.2mm}0
&&3\hspace*{2.2mm} ^-_2\hspace*{2.2mm}1\hspace*{2.2mm}^-_3\hspace*{2.2mm}0\hspace*{2.2mm} ^-_1\hspace*{2.2mm}2\hspace*{2.2mm}^-_3\hspace*{2.2mm}1\hspace*{2.2mm}^-_0\hspace*{2.2mm} 3\hspace*{2.2mm}^-_1\hspace*{2.2mm} 2\hspace*{2.2mm}^-_3\hspace*{2.2mm}0\hspace*{2.2mm}^-_1\hspace*{2.2mm}3 &&\\
_1|\hspace*{6mm}_2|\hspace*{6mm}_3|\hspace*{6mm}_0|\hspace{6mm}_1|
&&_1|\hspace*{24mm}_1|\hspace*{6mm}_2|\hspace*{24mm}_2|\hspace*{6mm}_1|&&\\
2\hspace*{2.2mm} ^-_0\hspace*{2.2mm}3\hspace*{7mm}0\hspace*{2.2mm} ^-_2\hspace*{2.2mm} 1\hspace{6.5mm}2
&&2\hspace*{2.2mm} ^-_1\hspace*{2.2mm}0\hspace*{2.2mm} ^-_3\hspace*{2.2mm}1\hspace*{2.2mm}^-_2\hspace*{2.2mm}3\hspace*{7mm}0\hspace*{2.2mm}^-_3\hspace*{2.2mm}2\hspace*{2.2mm} ^-_1\hspace*{2.2mm} 3\hspace*{2.2mm} ^-_0\hspace*{2.2mm} 1\hspace*{7mm}2&&\\
_3|\hspace*{6mm}_1|\hspace*{6mm}_1|\hspace*{6mm}_3|\hspace*{6.5mm}_3|
&&_3|\hspace*{6mm}_2|\hspace*{6mm}_0|\hspace*{6mm}_1|\hspace*{6mm}_1|\hspace*{6mm}_0|\hspace*{6mm}_2| \hspace*{6mm}_3| \hspace*{6mm}_3|&&\\
1\hspace*{2.2mm} ^-_2\hspace*{2.2mm}0\hspace*{7mm}3\hspace*{2.2mm}^-_0\hspace*{2.2mm}2\hspace*{7mm}1
&(^{\mbox{tor}}_\rightarrow)&1\hspace*{2.2mm}^-_0\hspace*{2.2mm}3\hspace*{2.2mm}^-_1\hspace{2.2mm}2\hspace*{2.2mm}^-_3\hspace*{2.2mm}0\hspace*{7mm}3\hspace*{2.2mm}^-_2\hspace*{2.2mm}1\hspace*{2.2mm} ^-_3\hspace*{2.2mm}0\hspace*{2.2mm}^-_1\hspace*{2.2mm} 2\hspace*{7mm}1&&\\
_0|\hspace*{6mm}_3|\hspace*{6mm}_2|\hspace*{6mm}_1|\hspace*{7mm}_0|
&&_2|\hspace*{24mm}_2|\hspace*{6mm}_0|\hspace*{24mm}_0|\hspace*{6mm}_2|&&\\
3\hspace*{2.2mm}^-_1\hspace*{2.2mm}2\hspace*{2.2mm} ^-_0\hspace*{2.2mm} 1\hspace*{2.2mm} ^-_3\hspace*{2.2mm} 0\hspace*{2.2mm}^-_2\hspace*{2.2mm}3
&&0\hspace*{2.2mm} ^-_3\hspace*{2.2mm}2\hspace*{2.2mm}^-_1\hspace*{2.2mm} 3\hspace*{2.2mm} ^-_0\hspace*{2.2mm} 1\hspace*{2.2mm} ^-_3\hspace*{2.2mm}2\hspace*{2.2mm}^-_1\hspace{2.2mm}0\hspace*{2.2mm} ^-_3\hspace*{2.2mm} 1\hspace*{2.2mm} ^-_2\hspace*{2.2mm} 3\hspace*{2.2mm} ^-_1\hspace*{2.2mm} 0&&\\
:\hspace*{36mm}:
&&:\hspace*{6mm}_0|\hspace*{6mm}_2|\hspace*{23.5mm}_2|\hspace*{6mm}_0|\hspace*{16mm}:&&\\
0\hspace*{2.2mm}^-_3\hspace*{2.2mm}1\hspace*{2.2mm} ^-_0\hspace*{2.2mm} 2\hspace*{2.2mm} ^-_1\hspace*{2.2mm} 3\hspace*{2.2mm}^-_2\hspace*{2.2mm}0
&&3\hspace*{2.2mm} ^-_2\hspace*{2.2mm}1\hspace*{2.2mm}^-_3\hspace*{2.2mm}0\hspace*{2.2mm} ^-_1\hspace*{2.2mm}2\hspace*{2.2mm}^-_3\hspace*{2.2mm}1\hspace*{2.2mm}^-_0\hspace*{2.2mm} 3\hspace*{2.2mm}^-_1\hspace*{2.2mm} 2\hspace*{2.2mm}^-_3\hspace*{2.2mm}0\hspace*{2.2mm}^-_1\hspace*{2.2mm}3&&\\
\end{array}\end{eqnarray}

In the lower section of display (\ref{oct6}), the left cutout is the same as in the upper section but  with the replication of the top horizontal path $0_-^31_-^02_-^13_-^20$ a layer below the botton horizontal path $3_-^12_-^01_-^30_-^23$. This is done in order to add four vertical compensatory edges on the corresponding right cutout at the same layer. In particular, this version of a $(_\rightarrow^{\times n})$ operation is made into a transformation of a planar graph into a toroidal graph, denoted by the indication $(_\rightarrow^{\mbox{tor}})$, by means of the additional parallel identification of the top and bottom horizontal paths.   
Note that the initial modifications on the right cutout consist in replacing the middle 4-belts with vertex-color dispositions $^{23}_{10}$ and $^{01}_{32}$ by the copies of $P_4\square P_2$ with vertex-color dispositions $^{2013}_{1320}$ and $^{0231}_{3102}$. This forces the mentioned addition of four edges and the vertex- and edge-coloring changes shown in the rest of the cutout.
\end{example}

Notice that the operation $(_\rightarrow^{\times n})$ of display (\ref{oct2}) in the proof of Theorem~\ref{t1} can be expressed in terms of the successive unfoldings in the statement of Theorem~\ref{starting}, so that Theorem~\ref{t1} itself could be expressed in terms of such unfoldings.

\begin{example}\label{ex2} 

Using Example~\ref{ex1} iteratively, the sequence of ETCed graphs in display (\ref{octaedro}) is engendered, where a transpose form of Definition~\ref{cutout} is employed (top and bottom identified),
In the cutout on the right in (\ref{octaedro}), three horizontal segments (representing two edges) are replaced by diaereses or transpose colons in order to illustrate that replacement of a pair of edges $e$ and $f$ with the same endvertex and edge colors (in this case 0 and 1 for the endvertices and 2 for the edges) with a new pair of edges with similar colors and forming a 4-cycle with $e$ and $f$ yields a new $\Gamma'$ with an ETGC. In this case, the new edges do not need the presence of a handle, since the two new edges partition a belt of $\Phi$, and the deletion of $e$ and $f$ does not create any change-of genus problem, either.
In larger cases of such planar graphs $\Gamma$, such pairs of edges may be realizable only via handles, line was the case of the upper-right cutout in display (\ref{oct6}) commented in Example~\ref{ex1}, which may increase the genus of the containing surface, which for planar graphs is a sphere, that has genus 0. This happens in the latter part of Example~\ref{sondos}, in which a toroidal graph is transformed into a graph of genus 2. 

\begin{eqnarray}\label{octaedro}\begin{array}{lllllc|c}
1\hspace*{2.2mm} _-^2\hspace*{2.2mm} 0&&1\hspace*{2.2mm} _-^2\hspace*{2.2mm} 0\hspace*{2.2mm} _-^3\hspace*{2.2mm} 2\hspace*{2.2mm} _-^1\hspace*{2.2mm} 3&&
1\hspace*{2.2mm} _-^2\hspace*{2.2mm} 0\hspace*{2.2mm} _-^3\hspace*{2.2mm} 2\hspace*{2.2mm} _-^1\hspace*{2.2mm} 3\hspace*{2.2mm}_-^0\hspace*{2.2mm}1\hspace*{2.2mm}_{..}^2\hspace{2.2mm}0&&\\
\!_3|\hspace*{6mm}_1|&&\!_3|\hspace*{6mm}_1|\hspace*{6mm}_0|\hspace*{6mm}_2|&&
\!_3|\hspace*{6mm}_1|\hspace*{6mm}_0|\hspace*{6mm}_2|\hspace*{6mm}_3|\hspace*{6mm}_1|&&\\
2\hspace*{2.2mm} ^-_0\hspace*{2.2mm}3&&2\hspace*{2.2mm} ^-_0\hspace*{2.2mm}3\hspace*{7mm} 1\hspace*{2.2mm} ^-_3\hspace*{2.2mm} 0&&
2\hspace*{2.2mm} ^-_0\hspace*{2.2mm}3\hspace*{7mm} 1\hspace*{2.2mm} ^-_3\hspace*{2.2mm} 0\hspace{6.5mm}2\hspace*{2.2mm} ^-_0\hspace*{2.2mm}3&&\\
\!_1|\hspace*{6mm}_2|&&\!_1|\hspace*{6mm}_2|\hspace*{6mm}_2|\hspace*{6mm}_1|&&
\!_1|\hspace*{6mm}_2|\hspace*{6mm}_2|\hspace*{6mm}_1|\hspace*{6.5mm}_1|\hspace*{6mm}_2|&&\\
0\hspace*{2.2mm} ^-_3\hspace*{2.2mm}1&\rightarrow&0\hspace*{2.2mm} ^-_3\hspace*{2.2mm}1\hspace*{7mm} 3\hspace*{2.2mm} ^-_0\hspace*{2.2mm} 2&\rightarrow&
0\hspace*{2.2mm} ^-_3\hspace*{2.2mm}1\hspace*{7mm} 3\hspace*{2.2mm} ^-_0\hspace*{2.2mm} 2\hspace*{7mm}0\hspace*{2.2mm} ^-_3\hspace*{2.2mm}1&&\\
\!_2|\hspace*{6mm}_0|&&\!_2|\hspace*{6mm}_0|\hspace*{6mm}_1|\hspace*{6mm}_3|&&
\!_2|\hspace*{6mm}_0|\hspace*{6mm}_1|\hspace*{6mm}_3|\hspace*{7mm}_2|\hspace*{6mm}_0|&&\\
3\hspace*{2.2mm} ^-_1\hspace*{2.2mm}2&&3\hspace*{2.2mm} ^-_1\hspace*{2.2mm}2\hspace*{2.2mm} ^-_3\hspace*{2.2mm} 0\hspace*{2.2mm} ^-_2\hspace*{2.2mm} 1&&
3\hspace*{2.2mm} ^-_1\hspace*{2.2mm}2\hspace*{2.2mm} ^-_3\hspace*{2.2mm} 0\hspace*{2.2mm} ^{..}_2\hspace*{2.2mm} 1\hspace*{2.2mm}^-_0\hspace*{2.2mm}3\hspace*{2.2mm}^-_1\hspace*{2.2mm}2&&0\hspace*{2.2mm}^{..}_2\hspace*{2.2mm}1\\
\!_0|\hspace*{6mm}_3|&&\!_0|\hspace*{23.5mm}_0|&&
\!_0|\hspace*{42mm}_3|&&\,|_2\hspace{6mm}|_2\\
1\hspace*{2.2mm} ^-_2\hspace*{2.2mm}0&&1\hspace*{2.2mm} ^-_2\hspace*{2.2mm}0\hspace*{2.2mm} ^-_3\hspace*{2.2mm} 2\hspace*{2.2mm} ^-_1\hspace*{2.2mm} 3&&
\,1\hspace*{2.2mm} ^-_2\hspace*{2.2mm}0\hspace*{2.2mm} ^-_3\hspace*{2.2mm} 2\hspace*{2.2mm} ^-_1\hspace*{2.2mm} 3\hspace*{2.2mm}^-_0\hspace*{2.2mm}1\hspace*{2.2mm}^{..}_2\hspace*{2.2mm}0&&1\hspace*{2.2mm}^{..}_2\hspace*{2.2mm}0\\
\end{array}\end{eqnarray}
\end{example}

\begin{example}\label{ex3}
The right cutout in display (\ref{oct2}) can be modified by replacing its 4-belt $H$ whose two vertical edges were deleted (and replaced by colons) in its representation on the left of display (\ref{oct3})
by the transpose $(\cdot)^t$ of the copy of $P_6\square P_2$ on the right of the display (after $\cup$), with its leftmost and rightmost edges (having degree-2 end-vertices) identified respectively to the corresponding horizontal edges of $H$.  
\begin{eqnarray}\label{oct3}\begin{array}{lllll}
0\hspace*{2.2mm} _-^3\hspace*{2.2mm} 1\hspace*{2.2mm} _-^0\hspace*{2.2mm} 2\hspace*{2.2mm} _-^1\hspace*{2.2mm} 3\hspace*{2.2mm} _-^2\hspace*{2.2mm}  0\hspace*{2.2mm} _-^3\hspace*{2.2mm} 1\hspace*{2.2mm} _-^0\hspace*{2.2mm} 2\hspace*{2.2mm} _-^1\hspace*{2.2mm} 3\hspace*{2.2mm} _-^2\hspace*{2.2mm} 0 &&
3\hspace*{2.2mm} _-^0\hspace*{2.2mm} 1\hspace*{2.2mm} _-^2\hspace*{2.2mm} 0\hspace*{2.2mm} _-^3\hspace*{2.2mm} 2\hspace*{2.2mm} _-^1\hspace*{2.2mm}  3\hspace*{2.2mm} _-^0\hspace*{2.2mm} 1&\\
\!_1|\hspace*{6mm}_2|\hspace*{6mm}_3|\hspace*{6mm}:\hspace*{8mm}:\hspace*{5mm}_2|\hspace*{6mm}_3|\hspace*{6mm}_0|\hspace*{6mm}_1|&\cup\;(&
\!_2|\hspace*{6mm}_3|\hspace*{6mm}_1|\hspace*{6mm}_0|\hspace*{6mm}_2|\hspace*{6mm}_3|&)^t\\
2\hspace*{2.2mm} ^-_0\hspace*{2.2mm}3\hspace*{2.2mm} ^-_1\hspace*{2.2mm} 0\hspace*{2.2mm} ^-_2\hspace*{2.2mm} 1\hspace*{2.2mm} ^-_3\hspace*{2.2mm} 2\hspace*{2.2mm} ^-_0\hspace*{2.2mm}3\hspace*{2.2mm} ^-_1\hspace*{2.2mm} 0\hspace*{2.2mm} ^-_2\hspace*{2.2mm} 1\hspace*{2.2mm} ^-_3\hspace*{2.2mm} 2 &&
0\hspace*{2.2mm} ^-_1\hspace*{2.2mm}2\hspace*{2.2mm} ^-_0\hspace*{2.2mm} 3\hspace*{2.2mm} ^-_2\hspace*{2.2mm} 1\hspace*{2.2mm} ^-_3\hspace*{2.2mm} 0\hspace*{2.2mm} ^-_1\hspace*{2.2mm}2&\\
\end{array}\end{eqnarray}
\end{example}

\subsection{Toroidal cases}\label{toroidal}

\begin{definition}\label{normal} A {\it normal cutout} of a connected toroidal graph $\Gamma$ is a closed rectangle $\Phi$ of $\mathbb{R}^2$ with $V(\Gamma)\subset\Phi\cap\mathbb{Z}^2$ and $E(\Gamma)$ given by vertical and horizontal unit-length segments such that identification of the left and right borders of $\Phi$ and identification of the top and bottom borders of $\Phi$ yield a representation of $\Gamma$, with the four corners of $\Phi$ representing a sole vertex of $\Gamma$ in $\mathbb{Z}^2$. 
\end{definition}

The identifications in Definition~\ref{normal} have plane representations in periodic extensions $\Phi'$ of the normal cutout $\Phi$ of $\Gamma$ as given in the following Definition~\ref{pets}, where belts delimit faces in $\Phi$ or $\Phi'$ as needed ahead.

\begin{definition}\label{pets} A {\it periodic extension} of a normal cutout $\Phi$ of a connected toroidal graph $\Gamma$ is a normal cutout formed by the union of two or more copies $\Phi_1,\Phi_2,\ldots,\Phi_n$ of $\Phi$, ($1<n\in\mathbb{Z}$), stacked horizontally, or vertically, along common linear vertical, or horizontal, borders, respectively, between each $\Phi_i$ and $\Phi_{i+1}$, ($1\le i<n$), that is: by identifying the right, or bottom, border of $\Phi_i$ with the left, or top, border of $\Phi_{i+1}$, respectively. 
\end{definition}

The truncated square tiling of the plane (obtainable online in various drawing versions) has the normal cutout $\Phi$ of the ETCed toroidal vertex-transitive 16-vertex cubic graph of girth 4 shown on the left of display (\ref{tess}), where truncated squares appear as 8-belts in $\Phi$ or in its periodic extensions. 
The following Definition~\ref{reyes} is exemplified by the accordion unfolding $\Phi'$ on the right of (\ref{tess}) of the normal cutout $\Phi$, shown on its left, with $\ell=2$ in this case.

\begin{eqnarray}\label{tess}\begin{array}{ll|ll}
0\hspace*{2.2mm}_-^1\hspace*{2.2mm}2\hspace*{2.2mm}_-^0\hspace*{2.2mm}3\hspace*{2.2mm}..\hspace*{2.2mm}1\hspace*{2.2mm}_-^3\hspace*{2.2mm}0&&&
0\hspace*{2.2mm}_-^1\hspace*{2.2mm}2\hspace*{2.2mm}_-^0\hspace*{2.2mm}3\hspace*{2.2mm}_-^2\hspace*{2.2mm}1\hspace*{2.2mm}_-^3\hspace*{2.2mm}
0\hspace*{2.2mm}..\hspace*{2.2mm}2\hspace*{2.2mm}_-^0\hspace*{2.2mm}3\hspace*{2.2mm}_-^2\hspace*{2.2mm}1\hspace*{2.2mm}_-^3\hspace*{2.2mm}0\\
:\hspace*{14.5mm}_2|\hspace*{6mm}_2|\hspace*{7mm}:&&&
:\hspace*{32mm}_1|\hspace*{6mm}_1|\hspace*{25mm}:\\
2\hspace*{2.2mm}_-^1\hspace*{2.2mm}0\hspace*{2.2mm}_-^3\hspace*{2.2mm}1\hspace*{7mm}3\hspace*{2.2mm}_-^0\hspace*{2.2mm}2&&&
2\hspace*{2.2mm}_-^1\hspace*{2.2mm}0\hspace*{2.2mm}_-^3\hspace*{2.2mm}1\hspace*{2.2mm}_-^2\hspace*{2.2mm}3\hspace*{2.2mm}_-^0\hspace*{2.2mm}
2\hspace*{7mm}0\hspace*{2.2mm}_-^3\hspace*{2.2mm}1\hspace*{2.2mm}_-^2\hspace*{2.2mm}3\hspace*{2.2mm}_-^0\hspace*{2.2mm}2\\
\!_3|\hspace*{6mm}_2|\hspace*{6mm}_0|\hspace*{6mm}_1|\hspace*{6mm}_3|&&&
\!_3|\hspace*{6mm}_2|\hspace*{6mm}_0|\hspace*{6mm}_1|\hspace*{7mm}
\!_3|\hspace*{6mm}_2|\hspace*{6mm}_0|\hspace*{6mm}_1|\hspace*{6mm}_3|\\
1\hspace*{7mm}3\hspace*{2.2mm}_-^1\hspace*{2.2mm}2\hspace*{2.2mm}_-^3\hspace*{2.2mm}0\hspace*{2.2mm}_-^2\hspace*{2.2mm}1&&&
1\hspace*{7mm}3\hspace*{2.2mm}_-^1\hspace*{2.2mm}2\hspace*{2.2mm}_-^3\hspace*{2.2mm}0\hspace*{2.2mm}_-^2\hspace*{2.2mm}
1\hspace*{2.2mm}_-^0\hspace*{2.2mm}3\hspace*{2.2mm}_-^1\hspace*{2.2mm}2\hspace*{2.2mm}_-^3\hspace*{2.2mm}0\hspace*{2.2mm}_-^2\hspace*{2.2mm}1\\
\!_0|\hspace*{6mm}_0|\hspace*{23.5mm}_0|&&&
\!_0|\hspace*{6mm}_0|\hspace*{60mm}_0|\\
3\hspace*{7mm}1\hspace*{2.2mm}_-^2\hspace*{2.2mm}0\hspace*{2.2mm}_-^3\hspace*{2.2mm}2\hspace*{2.2mm}_-^1\hspace*{2.2mm}3&&&
3\hspace*{7mm}1\hspace*{2.2mm}_-^2\hspace*{2.2mm}0\hspace*{2.2mm}_-^3\hspace*{2.2mm}2\hspace*{2.2mm}_-^1\hspace*{2.2mm}
3\hspace*{2.2mm}_-^0\hspace*{2.2mm}1\hspace*{2.2mm}_-^2\hspace*{2.2mm}0\hspace*{2.2mm}_-^3\hspace*{2.2mm}2\hspace*{2.2mm}_-^1\hspace*{2.2mm}3\\
\!_2|\hspace*{6mm}_3|\hspace*{6mm}_1|\hspace*{6mm}_0|\hspace*{6mm}_2|&&&
\!_2|\hspace*{6mm}_3|\hspace*{6mm}_1|\hspace*{6mm}_0|\hspace*{7mm}
\!_2|\hspace*{6mm}_3|\hspace*{6mm}_1|\hspace*{6mm}_0|\hspace*{6mm}_2|\\
0\hspace*{2.2mm} _-^1\hspace*{2.2mm}2\hspace*{2.2mm}_-^0\hspace*{2.2mm}3\hspace*{2.2mm}..\hspace*{2mm}1\hspace*{2.2mm}_-^3\hspace*{2.2mm}0&&&
0\hspace*{2.2mm} _-^1\hspace*{2.2mm}2\hspace*{2.2mm}_-^0\hspace*{2.2mm}3\hspace*{2.2mm}_-^2\hspace*{2mm}1\hspace*{2.2mm}_-^3\hspace*{2.2mm}
0\hspace*{2.5mm} ..\hspace*{2.5mm}2\hspace*{2.2mm}_-^0\hspace*{2.2mm}3\hspace*{2.2mm}_-^2\hspace*{2mm}1\hspace*{2.2mm}_-^3\hspace*{2.2mm}0\\
\end{array}\end{eqnarray}

\begin{definition}\label{reyes}
Given a normal cutout $\Phi$ of a toroidal cubic graph $\Gamma$ of girth 4, an ({\it accordion}) {\it unfolding} of $\Phi$ is a cutout $\Phi'$ of a toroidal cubic graph $\Gamma'$ obtained by the replacement of 4-belts of the form $P_2\square P_2$ by copies of $P_2\square P_{2\ell}$, where $1<\ell\in\mathbb{Z}$.
\end{definition}

\begin{remark}\label{remark}
The leftmost cutout $\Phi$ in display (\ref{tess}) is related to the middle cutout in display (\ref{octaedro}) as follows. The leftmost vertical path of $\Phi$ equals the rightmost vertical path and is joined with the second leftmost path by two horizontal edges colored $0\hspace*{2.2mm}^1_-\hspace*{2.2mm}2$ and $2\hspace*{2.2mm}^1_-\hspace*{2.2mm}0$. We apply {\it cycle exchange} by replacing those two edges by the other two edges forming an auxiliary square $K_2\square K_2$ in the plane of  $\Phi$ but resulting in a cutout $\Phi'$ equivalent to the middle cutout $\Phi''$ in display (\ref{octaedro}) (with the leftmost vertical path $[1_-^32_-^10_-^23_-^01]^t$ of $\Phi''$ obtained as the second leftmost vertical path $[2_-^10_-^23_-^01_-^32]^t$ of $\Phi'$) of a planar graph by the now apparently separated modified leftmost path $[0_-^12_-^31_-^03_-^20]^t$ of $\Phi'$, since it is already present as the  now modified rightmost path of $\Phi'$. 

Similar cycle exchanges were given in the upper-right cutout of display (\ref{oct6}) for Example~\ref{ex1} and the right cutout of display (\ref{octaedro}) for Example~\ref{ex2}, where horizontal-edge pairs indicated by diaeresis pairs are to be replaced by the other two edges in the auxiliary squares $K_2\square K_2$ shown to the right. Another case of cycle exchange will appear in Theorem~\ref{tu} to increase the genus of the considered cubic graphs of genus 4. 

Similarly, the operation $(^{\mbox{tor}}_{\rightarrow})$ in display (\ref{oct6}) corresponds to an unfolding as in Definition~\ref{unfold}, where two compensation vertical edges are added to the original external face of a planar graph making it into a toroidal graph.
\end{remark}   

\begin{theorem}\label{16vertex}
A toroidal vertex-transitive cubic graph $\Gamma_1^1$ of girth 4 on $4\ell$ vertices, with $\ell=4$ and normal cutout $\Phi$ as on the left of display (\ref{tess}), or $4\le\ell\in\mathbb{Z}$ and normal cutout $\Psi$ as in the lower-right of display (\ref{oct6}) via operation $(_\rightarrow^{\mbox{\rm{tor}}})$ has an ETGC, as do all the toroidal vertex-transitive graphs $\Gamma_h^k$ on $4\ell hk$ vertices obtained from $\Phi$ or $\Psi$ by periodic extensions, both horizontally $h$ times and vertically $k$ times ($0<h,k\in\mathbb{Z}$). All 4-belts of the graphs $\Gamma_h^k$ and their limit plane tessellations have both their vertex sets and edge sets in bijective correspondence with the 4-color set $[4]$. Each 8-belt of $\Gamma_h^k$ has its antipodal elements (vertices, edges) with a common color, such that the 4 colors are employed, each twice.
No other TCs in such graphs exist but those ETGCs.
\end{theorem}

\begin{proof}
By assigning color numbers in the set [4] to the vertices and edges of $\Gamma_1^1$ as on the left of display (\ref{tess}) or the lower-right of display (\ref{oct6}), one such claimed ETGCs is obtained. This generalizes by double continuation for all graphs $\Gamma_h^k$. Clearly, 
no other TCs exists in them. 
\end{proof}

\begin{theorem}\label{sonidos}
Starting with a normal cutout $\Phi$ of a toroidal cubic graph $\Gamma$ of girth 4 as in the statement of Theorem~\ref{starting},  
successive unfoldings and periodic extensions yield normal cutouts $\Phi'$ of toroidal cubic graphs $\Gamma'$ lacking $\ell$-belts with $\ell\not\equiv 0 \mod 4$ that have ETGCs. This  produces an infinite collection of such graphs $\Gamma'$..
\end{theorem}

\begin{proof}
The same considerations given en the proof of Theorem~\ref{starting} leads to the completion of the proof, but taking care of the lattice containment of each resulting graph $\Gamma'$.
\end{proof}

\begin{example}\label{sondos}
The left case in display (\ref{te}) is obtained from $Q_3$ drawn via a normal cutout $\Phi$ of $Q_3$,   
but note here that the two shown 8-cycles in such $\Phi$ are not really 8-belts; 
as a result, vertically-stacked periodic extensions of such normal cutout do not work properly, as is the case of normal cutouts of toroidal graphs; in fact, one instance of such normal-cutout extension leads to a normal cutout unrelated to that of $Q_3$, with an ETGC equivalent to that on the left of display (\ref{tess}), but now represented as the {\it tilted cutout} on the right of (\ref{te}), where the top and bottom borders are to be identified (with no horizontal displacements) and the right border is to be identified two layers below the left border, indicated by colons and semicolons on such borders. Notice that a single horizontal periodic extension of such tilted cutout yields a normal cutout.

\begin{eqnarray}\label{te}\begin{array}{ll|ll}
0\hspace*{2.2mm}_-^3\hspace*{2.2mm}1\hspace*{2.2mm}_-^0\hspace*{2.2mm}2\hspace*{2.2mm}_-^1\hspace*{2.2mm}3\hspace*{2.2mm}_-^2\hspace*{2.2mm}0&&&
0\hspace*{2.2mm}_-^3\hspace*{2.2mm}1\hspace*{2.2mm}_-^2\hspace*{2.2mm}3\hspace*{2.2mm}_-^1\hspace*{2.2mm}0\hspace*{2.2mm}_-^3\hspace*{2.2mm}2\\
\!_1|\hspace*{6mm}_2|\hspace*{23.5mm}_1|&&&
\!_2|\hspace*{6mm}_0|\hspace*{23.5mm}_0|\\
2\hspace*{2.2mm} _-^0\hspace*{2.2mm} 3\hspace*{2.2mm}_-^1\hspace*{2.2mm}0\hspace*{2.2mm}_-^2\hspace*{2.2mm}1\hspace*{2.2mm}_-^3\hspace*{2.2mm}2&&&
3\hspace*{2.2mm}_-^1\hspace*{2.2mm}2\hspace*{2.2mm}_-^3\hspace*{2.2mm}0\hspace*{2.2mm}_-^1\hspace*{2.2mm}3\hspace*{2.2mm}_-^2\hspace*{2.2mm}1\\
:\hspace*{15mm}_3|\hspace*{6mm}_0|\hspace*{6mm}:&&&
:\hspace*{15mm}_2|\hspace*{6mm}_0|\hspace*{7.5mm};\\
0\hspace*{2.2mm}_-^3\hspace*{2.2mm}1\hspace*{2.2mm}_-^0\hspace*{2.2mm}2\hspace*{2.2mm}_-^1\hspace*{2.2mm}3\hspace*{2.2mm}_-^2\hspace*{2.2mm}0&&&
2\hspace*{2.2mm}_-^1\hspace*{2.2mm}3\hspace*{2.2mm}_-^0\hspace*{2.2mm}1\hspace*{2.2mm}_-^3\hspace*{2.2mm}2\hspace*{2.2mm}_-^1\hspace*{2.2mm}0\\
&&&\!_0|\hspace*{6mm}_2|\hspace*{23.5mm}_2|\\
&&&1\hspace*{2.2mm}_-^3\hspace*{2.2mm}0\hspace*{2.2mm}_-^1\hspace*{2.2mm}2\hspace*{2.2mm}_-^3\hspace*{2.2mm}1\hspace*{2.2mm}_-^0\hspace*{2.2mm}3\\
&&&;\hspace*{15mm}_0|\hspace*{6mm}_2|\hspace*{7mm}:\\
&&&0\hspace*{2.2mm}_-^3\hspace*{2.2mm}1\hspace*{2.2mm}_-^2\hspace*{2.2mm}3\hspace*{2.2mm}_-^1\hspace*{2.2mm}0\hspace*{2.2mm}_-^3\hspace*{2.2mm}2\\
\end{array}\end{eqnarray}

On the other hand, successive replacement of $P_2\square P_2$ by $P_2\square P_4$ on the mentioned left normal cutout in (\ref{te}) yields the left normal cutout of (\ref{tessa}), while the right case of display (\ref{tessa}) is obtained from the lower-right of display (\ref{oct6}), again via successive unfoldings. 
 For a novel example in a surface of genus 2, two vertical segments in this normal cutout representing edges $e$ and $f$ with common endvertex and edge colors (in this case 0 and 2 for the endvertices and 1 for the edges) are replaced by corresponding colons in order to illustrate the generation of another $\Gamma'$ with a new pair of edges with similar color structure and forming a 4-cycle with $e$ and $f$, which are therefore deleted to yield an ETGC. The resulting exchange 4-cycle is drawn on the lower-right of display (\ref{tessa}). Since this operation is done on a toroidal graph via the addition of a handle to the torus to trace on it the new edges, the resulting $\Gamma$ is a graph only embeddable in the oriented surface 2. Note the deletion of $e$ and $f$ leaves the resulting modified cutout with two new cycles of lengths 6 and 14 that must be taken as the borders of the mentioned handle. Compare with the upper-right of display (\ref{oct6}) and lower-right of display (\ref{octaedro}) in Examples~\ref{ex1} and~\ref{ex2}, respectively.

\begin{eqnarray}\label{tessa}\begin{array}{ll|llc|c}
0\hspace*{2.2mm} _-^3\hspace*{2.2mm} 1\hspace*{2.2mm} _-^0\hspace*{2.2mm} 2\hspace*{2.2mm} _-^1\hspace*{2.2mm} 3\hspace*{2.2mm} _-^2\hspace*{2.2mm}0
&&&0\hspace*{2.2mm}_-^3\hspace*{2.2mm}1\hspace*{2.2mm}_-^2\hspace*{2.2mm}3\hspace*{2.2mm}_-^0\hspace*{2.2mm}2\hspace*{2.2mm}_-^3\hspace*{2.2mm}  1\hspace*{2.2mm}_-^2\hspace*{2.2mm}0\hspace*{2.2mm}_-^3\hspace*{2.2mm}2\hspace*{2.2mm}_-^1\hspace*{2.2mm} 3\hspace*{2.2mm} _-^2\hspace*{2.2mm} 0 &&\\ 
:\hspace*{15mm}_3|\hspace*{6mm}_0|\hspace*{5.6mm}:
&&&\!_1|\hspace*{24mm}_1\!:\hspace*{4mm}_0|\hspace*{24mm}_0|\hspace*{6mm}_1|&&\\
:\hspace*{16.2mm} 0\hspace*{2.2mm} ^-_2\hspace*{2.2mm}1\hspace*{5.2mm}:
&&&2\hspace*{2.2mm} ^-_0\hspace*{2.2mm}3\hspace*{2.2mm}^-_2\hspace*{2.2mm} 1\hspace*{2.2mm} ^-_3\hspace*{2.2mm} 0\hspace*{7mm} 3\hspace*{2.2mm} ^-_1\hspace*{2.2mm}2\hspace*{2.2mm}^-_3\hspace{2.2mm}1\hspace*{2.2mm} ^-_2\hspace*{2.2mm} 1\hspace*{7mm}2&& \\
:\hspace*{15mm}_1|\hspace*{6mm}_3|\hspace*{5.6mm}:
&&&\!_3|\hspace*{6mm}_1|\hspace*{6mm}_0|\hspace*{6mm}_2|\hspace*{6mm}_2|\hspace*{6mm}_0|\hspace*{6mm}_1|\hspace*{6mm}_3|\hspace*{7mm}_3|&&\\
:\hspace*{16.2mm} 3\hspace*{2.2mm} ^-_0\hspace*{2.2mm} 2\hspace*{5.2mm}:
&&&1\hspace*{2.2mm} ^-_2\hspace*{2.2mm}0\hspace*{2.2mm}^-_3\hspace*{2.2mm}2\hspace*{2.2mm} ^-_1\hspace*{2.2mm} 3\hspace*{7mm} 0\hspace*{2.2mm} ^-_3\hspace*{2.2mm}1\hspace*{2.2mm}^-_2\hspace*{2.2mm} 3\hspace*{2.2mm} ^-_0\hspace*{2.2mm} 2\hspace*{7mm}1&& \\
:\hspace*{15mm}_2|\hspace*{6mm}_1|\hspace*{5.6mm}:
&&&\!_0|\hspace*{24mm}_0|\hspace*{6mm}_1|\hspace*{24mm}_1|\hspace*{6mm}_0|&&\\
3\hspace*{2.2mm} ^-_1\hspace*{2.2mm}2\hspace*{2.2mm} ^-_0\hspace*{2.2mm} 1\hspace*{2.2mm} ^-_3\hspace*{2.2mm} 0\hspace*{2.2mm} ^-_2\hspace*{2.2mm}3
&&&3\hspace*{2.2mm} ^-_1\hspace*{2.2mm}2\hspace*{2.2mm} ^-_3\hspace*{2.2mm} 0\hspace*{2.2mm} ^-_2\hspace*{2.2mm} 1\hspace*{2.2mm} ^-_3\hspace*{2.2mm} 2\hspace*{2.2mm} ^-_0\hspace*{2.2mm}3\hspace*{2.2mm} ^-_2\hspace*{2.2mm} 1\hspace*{2.2mm} ^-_3\hspace*{2.2mm} 0\hspace*{2.2mm} ^-_2\hspace*{2.2mm} 3&& \\ 
\!_0|\hspace*{6mm}_3|\hspace*{24mm}_0|
&&&\,:\hspace*{5.2mm}_0|\hspace*{6mm}_1|\hspace*{24mm}_1|\hspace*{6mm}_0|\hspace*{16mm}:&&\\
1\hspace*{2.2mm} ^-_2\hspace*{2.2mm}0\hspace*{25mm} 1
&&&0\hspace*{2.2mm} ^-_3\hspace*{2.2mm}1\hspace*{7mm} 3\hspace*{2.2mm} ^-_0\hspace*{2.2mm} 2\hspace*{2.2mm}^-_3\hspace*{2.2mm} 1\hspace*{2.2mm} ^-_2\hspace*{2.2mm}0\hspace*{7mm} 2\hspace*{2.2mm} ^-_1\hspace*{2.2mm} 3\hspace*{2.2mm}^-_2\hspace*{2.2mm}1&& \\
\!_3|\hspace*{6mm}_1|\hspace*{24mm}_3|
&&&\!_1|\hspace*{6mm}_2|\hspace*{6mm}_2|\hspace*{6mm}_1\!:\hspace*{4mm}_0|\hspace*{6mm}_3|\hspace*{6mm}_3|\hspace*{6mm}_0|\hspace*{6mm}_1|&&\\
2\hspace*{2.2mm} ^-_0\hspace*{2.2mm}3\hspace*{25mm}2 
&&&2\hspace*{2.2mm} ^-_0\hspace*{2.2mm}3\hspace*{7mm} 1\hspace*{2.2mm} ^-_3\hspace*{2.2mm} 0\hspace*{2.2mm}^-_2\hspace*{2.2mm} 3\hspace*{2.2mm} ^-_1\hspace*{2.2mm}2\hspace*{7mm} 0\hspace*{2.2mm} ^-_2\hspace*{2.2mm} 1\hspace*{2.2mm}^-_3\hspace{2.2mm}3&&2\hspace*{2.2mm}_-^1\hspace*{2.2mm}0\\
\!_1|\hspace*{6mm}_2|\hspace*{24mm}_1|
&&&\,:\hspace*{5.2mm}_1|\hspace*{6mm}_0|\hspace*{24mm}_0|\hspace*{6mm}_1|\hspace*{15.5mm}:&&_1\!:\hspace*{6mm}_1\!:\\
0\hspace*{2.2mm} _-^3\hspace*{2.2mm} 1\hspace*{2.2mm} _-^0\hspace*{2.2mm} 2\hspace*{2.2mm} _-^1\hspace*{2.2mm} 3\hspace*{2.2mm} _-^2\hspace*{2.2mm}0
&&&1\hspace*{2.2mm} ^-_2\hspace*{2.2mm}0\hspace*{2.2mm} ^-_3\hspace*{2.2mm} 2\hspace*{2.2mm} ^-_1\hspace*{2.2mm} 3\hspace*{2.2mm} ^-_2\hspace*{2.2mm} 0\hspace*{2.2mm} ^-_3\hspace*{2.2mm}1\hspace*{2.2mm} ^-_2\hspace*{2.2mm} 3\hspace*{2.2mm} ^-_0\hspace*{2.2mm} 2\hspace*{2.2mm} ^-_3\hspace*{2.2mm} 1 &&0\hspace*{2.2mm}^-_1\hspace*{2.2mm}2\\
\end{array}\end{eqnarray}

\end{example}

\subsection{Algorithmic aspects of ETGCs}\label{alg}

\begin{definition}\label{sabado}
Given a cutout or normal cutout $\Phi$ of a planar or toroidal, respectively, cubic graph $\Gamma$ of girth 4 with all its $\ell$-belts having $\ell\equiv 0 \mod 4$, and given a vertex $v$ of $\Gamma$ incident to edges $e=(v,v_e),f=(v,v_f)$ and $g=(v,v_g)$, the coloring operation {\it ETCing} (for {\it efficient-total coloring}) by means of the color set $\{c_0,c_1,c_2,c_3\}=\{0,1,2,3\}=[4]$ at the vertex $v$ is given as follows:
\begin{enumerate}\item if $v,e,f$ are attributed colors $c_0,c_1,c_2$, respectively, then $g$ is assigned color $c_3$;
\item if $v,v_e,v_f$ are attributed colors $c_0,d_1,d_2$, respectively, where $[4]=\{c_0,d_1,d_2,d_3\}$, then $e_g$ is assigned color $d_3$;
\item each belt $H$ in $\Phi$ is attributed colors periodically: $$(c_0,d_0,c_1,d_1,c_2,d_2,c_3,d_3,\ldots,c_0,d_0,c_1,d_1,c_2,d_2,c_3,d_3)$$ where the $c_i$ are the colors for the successive vertices of $H$ and the $d_i$ are the colors for the edges between those successive vertices.
\end{enumerate}
\end{definition}

\begin{theorem}\label{tt} Let $\Phi$, $\Gamma,$ $v,e,f$ and $g$ be as assumed in Definition~\ref{sabado} and let $v$ belong to a 4-belt $H_0$ of $\Phi$.
Initializing by coloring $\{v,e,f,g\}$ in one-to-one correspondence with $[4]$ and continuing by coloring the remaining vertices of $H_0$, e.g.
 as in either case of display (\ref{algo}),
\begin{eqnarray}\label{algo}\begin{array}{llllll}
2\hspace*{2.2mm}_-^1\hspace*{2.2mm}0\hspace*{2.2mm}_-^3\hspace*{2.2mm}1\hspace*{2.2mm}_-^2\hspace*{2.2mm}3
&&&&2\hspace*{2.2mm}_-^3\hspace*{2.2mm}0\hspace*{2.2mm}_-^2\hspace*{2.2mm}1\hspace*{2.2mm}_-^0\hspace*{2.2mm}3\\
\hspace*{8.5mm}\!_2|\hspace*{6mm}_0|&&\mbox{or}&&\hspace*{8.5mm}\!_1|\hspace*{6mm}_3|\\
1\hspace*{2.2mm}_-^0\hspace*{2.2mm}3\hspace*{2.2mm}_-^1\hspace*{2.2mm}2\hspace*{2.2mm}_-^3\hspace*{2.2mm}0
&&&&1\hspace*{2.2mm}_-^2\hspace*{2.2mm}3\hspace*{2.2mm}_-^0\hspace*{2.2mm}2\hspace*{2.2mm}_-^1\hspace*{2.2mm}0\\
\end{array}\end{eqnarray} 
 and surrounding vertices via items 1 and 2 and neighboring belts via item 3, 
 forced continuation via the ETCing operation allows to obtain an ETGC in $\Gamma$ by its completion in $\Phi$. Moreover, as in display~\ref{algo}, there are two ETCs on $\Gamma$ over a common VC of $V(\Gamma)$ with no common color of such ETCs on each fixed edge of $\Gamma$. The two resulting ETCs guarantee the existence of an EGC on the prism $\Gamma\square P_2$.
\end{theorem}

\begin{proof}
Starting at any 4-belt of $\Phi$ that is ETCed, as in both examples in display (\ref{algo}) (under a common vertex coloring),  allows a forced continuation of the ETCing operation, as shown partially in display (\ref{algo}) at both sides of each of the two 4-belts. Since the $\ell$-belts of $\Gamma$ have $\ell\equiv 0 \mod 4$, continuation of ETCing is fulfilled at each such belt by completing periodically the coloring of edges and vertices,
 exemplified by the TCs of  the belts in displays (\ref{oct6}), (\ref{octaedro}), (\ref{oct3}) and (\ref{tessa}). This produces a forced ETGC by just starting ETCing a single 4-belt.
Since there are two ETCs for $\Gamma$ for a common TC of $\Gamma$ the common vertex coloring is used to assigned those common colors to the edges of the form $(v,0)(v,1)$ in the prism $\Gamma\square P_2$, where $V(P_2)=\{0,1\}$. 
\end{proof}

\subsection{Raising the genus via handles}\label{hand}

\begin{theorem}\label{tu}
Assume that $\Phi$, $\Gamma$ and an ETGC $\Psi$ of $\Gamma$ obtained via $\Phi$ are produced by means of Theorem~\ref{tt}. 
Given two nonadjacent edges $e=(v,v')$ and $f=(w,w')$ of $\Gamma$ in $\Phi$ with the same vertex-and-edge color pattern in $\Psi$, e.g., $\Psi(v)=\Psi(w)$, $\Psi(v')=\Psi(w')$ and $\Psi(e)=\Psi(f)$, consider the 4-cycle formed by $e$, $f$, $g=(v,w)$ and $h=(v',w')$. Then, by replacing $e$ and $f$ with $g$ and $h$ in $\Gamma$, resulting in a new cubic graph $\Gamma'=(\Gamma\setminus\{e,f\})\cup\{g,h\}$, and keeping all vertex colors and remaining edge colors while adopting new colors $\Psi(g)=\Psi(h)$ equal to old colors $\Psi(e)=\Psi(f)$, an ETGC of $\Gamma'$ is obtained. If $e$ and $f$ separate pairs of belts $\{H_e,H'_e\}$ and $\{H_f,H'_f\}$, respectively, where $H_e,H'_e,H_f,H'_f$ are four different belts of $\Phi$, then the genus of $\Gamma'$ is one more than the genus of $\Gamma$.
\end{theorem}

\begin{proof}
If $e$ and $f$ belong to a common belt of $\Gamma$, then $\Gamma'$ still has all its $\ell$-belts with $\ell\equiv 0 \mod 4$. Adjacent belts, i.e. those having at least an edge in common, have different vertex-color sequences, as in item 3 of Definition~\ref{sabado}, so they do not offer pairs of edges $e$ and $f$ as in the statement. If $e$ and $f$ separate pairs of belts $\{H_e,H'_e\}$ and $\{H_f,H'_f\}$, respectively, where $H_e,H'_e,H_f,H'_f$ are four different belts of $\Phi$, then the substitution of the pair $\{e,f\}$
with the pair $\{g,h\}$ transforms the pairs $\{H_e,H'_e\}$ and $\{H_f,H'_f\}$ into single cycles $H''_e$ and $H''_f$ of lengths $|H_e|+|H'_e|-2$ and $|H_f|+|H'_f|-2$, respectively, which are congruent to 2 mod 4. The new cycles $H''_e$ and $H''_f$ in $\Phi$ must be interpreted as the borders of a handle containing the new edges $g$ and $h$. With this, $\Gamma'$ is seen to have genus equal to genus of $\Gamma$ plus one.
\end{proof}

One can also reduce by one unit the genus of a $\Gamma$. For example, there are two horizontal edges colored $0\hspace*{2.2mm}^1_-\hspace*{2.2mm}2$ and $2\hspace*{2.2mm}^1_-\hspace*{2.2mm}0$ on the left of display~\ref{tess} with $0\hspace*{2.2mm}^1_-\hspace*{2.2mm}2$ repeated on top and bottom. By replacing them with a pair of vertical edges on the same endvertex quadruple, we pass from the original toroidal graph into a planar graph again, both of them ETCed.

\begin{corollary}\label{tudo} For $1<h\in\mathbb{Z}$,
the procedure of Theorem~\ref{tu} is applicable iteratively starting at any adequately large cutout $\Phi$ of a planar, or normal cutout $\Phi$ of a toroidal, cubic graph $\Gamma_0$ of girth 4, with $h$ pairs of edges separating pairwise different belt pairs, each such pair having its two edges sharing a vertex-edge color disposition. The resulting iteration  generates a sequence $\Gamma_0,\Gamma_1,\Gamma_2,\ldots,\Gamma_h$ of cubic graphs $\Gamma_i$ of girth 4 with two different ETGCs on a common TC, where the genus of each $\Gamma_i$ is one larger than the genus of $\Gamma_{i-1}$, for $0<i\le h$. Moreover, the prism $P_2\square\Gamma_i$ has an EGC based on those two ETGCs.
\end{corollary}

\begin{proof}
Straightforward, based on the iteration of cycle exchanges, with the last assertion based on an argument similar to that of Theorem~\ref{starting}.
\end{proof}

\begin{conjecture}\label{con1} Let $\Gamma$ be a finite connected simple cubic graph of girth 4. Then every ETC of
$\Gamma$ is obtained via Corollary~\ref{tudo} departing from $Q_3$ (as in Theorem~\ref{t1}) via the following four constructive operations:
periodic extensions (Definitions~\ref{pe} and~\ref{pets}),
accordion unfoldings (Definitions~\ref{unfold} and~\ref{reyes}), cycle exchanges (Remark~\ref{remark}) and
ETCings (Definition~\ref{sabado}). 
\end{conjecture}

\begin{example}
Three examples of Theorem~\ref{tu} were already presented in the upper right of display (\ref{oct6}), the right of display (\ref{cutout}) and the right of display (\ref{tessa}), in the context of examples~\ref{ex1}, \ref{ex2} and \ref{sondos}, respectively.
\end{example}

\subsection{Edge-partitions into 3-paths and into 3-stars}\label{chance}

\begin{theorem}
Each ETGC $\Phi$ of a cubic graph $\Gamma$ of girth 4 whose $\ell$-belts have $\ell\equiv 0 \mod 4$ and is either planar or toroidal insures an edge-partition of $\Gamma$ into $\frac{1}{2}|V(\Gamma)|$ paths $P_4$ of length 3. Moreover, the total number of such partitions is $6|V(\Gamma)|$.
\end{theorem}

\begin{proof}
The color set $[4]$ yields pairs 
$\{0123,1302\}$, $\{0132,1203\}$, $\{0213,2301\}$, $\{0231,2103\}$, $\{0312,3201\}$ and $\{0321,3102\}$ of 4-color sequences. Each such pair 
splits $E(\Gamma)$ into $\frac{1}{2}|V(\Gamma)|$ paths, half respecting each color sequence. By following for example the vertices colored successively 0123 or 1302, distinguishing the resulting paths $A,B,C,D,\ldots$ etc., each $\Gamma$ as in the statement gets a partition as claimed. So, the 3-cube gets the cutout in display (\ref{octstars}):
\begin{eqnarray}\label{octstars}\begin{array}{l}
0\hspace*{2.2mm} _-^A\hspace*{2.2mm} 1\hspace*{2.2mm} _-^A\hspace*{2.2mm} 2\hspace*{2.2mm} _-^A\hspace*{2.2mm} 3\hspace*{2.2mm} _-^B\hspace*{2.2mm} 0\\
|_B\hspace*{5mm}|_D\hspace*{5mm}|_D\hspace*{5mm}|_B\hspace*{5mm}|_B\\
2\hspace*{2.2mm}_C^-\hspace*{2.2mm}3\hspace*{2.2mm} _D^-\hspace*{2.2mm}0\hspace*{2.2mm}_C^-\hspace*{2.2mm} 1\hspace*{2.2mm} _C^-\hspace*{2.2mm} 2
\end{array}\end{eqnarray} 
A similar treatment holds for any other cubic graph of girth 4 and $\ell$-belts with $\ell\equiv 0 \mod 4$, either planar or toroidal. \end{proof}

\begin{theorem}
Each ETGC $\Phi$ of a cubic graph $\Gamma$ of girth 4 whose $\ell$-belts have $\ell\equiv 0 \mod 4$ and is either planar or toroidal insures an edge-partition of $\Gamma$ into $\frac{1}{4}|V(\Gamma)|$ 3-stars $K_{1,3}$. Moreover, the total number of such partitions is 4 and the centers of the 3-stars of these partitions form a partition of $|V(\Gamma)|$ into the colors $0$, $1$, $2$ and $3$.
\end{theorem}

\begin{proof}
This is direct conclusion of item (b) of Definition~\ref{ahora}.
\end{proof}

\subsection{4-regular graphs of girth 5 with TCs but not ETCs}\label{7}

\begin{theorem}\label{Gamma}
The edge-disjoint union $\Gamma=Pet^2$ of two pentagons $P_0=(v_0v_4v_8v_{12}v_{16})$ and $P_1=(v_2v_6v_{10}v_{14}v_{18})$, two pentagrams $Q_0=(v_1v_5v_9v_{13}v_{17})$ and $Q_1=(v_3v_7v_{11}v_{15}v_{19})$ and the Hamilton cycle $\Sigma=(v_0v_1v_2\cdots v_{17}v_{18}v_{19})$ is a 20-vertex 4-regular graph of girth 5 with a TC that is not efficient. Moreover, $E(\Sigma)$ decomposes as:  
$$\begin{array}{cc}F_0^0=\{v_0v_1,v_4v_5,v_8v_9,v_{12}v_{13},v_{16}v_{17}\},&
F_0^1=\{v_0v_{19},v_4v_3,v_8v_7,v_{12}v_{11},v_{16}v_{15}\},\\
F_1^0=\{v_2v_1,v_6v_5,v_{10}v_9,v_{14}v_{13},v_{18}v_{17}\},&
F_1^1=\{v_2v_3,v_6v_7,v_{10}v_{11},v_{14}v_{15},v_{18}v_{19}\},\end{array}$$
so that $P_i\cup Q_j\cup F_i^j$ is a copy of $Pet$ in $\Gamma$, for $i=0,1$ and $j=0,1$.
$\Gamma$ contains exactly 54 5-cycles, but the TC has just 14 5-cycles with their vertex and edge sets in bijective correspondence with the color set $[5]$. The remaining 40 5-cycles do not have such bijective correspondence.
\end{theorem}

\begin{figure}[htp]
\includegraphics[scale=0.882]{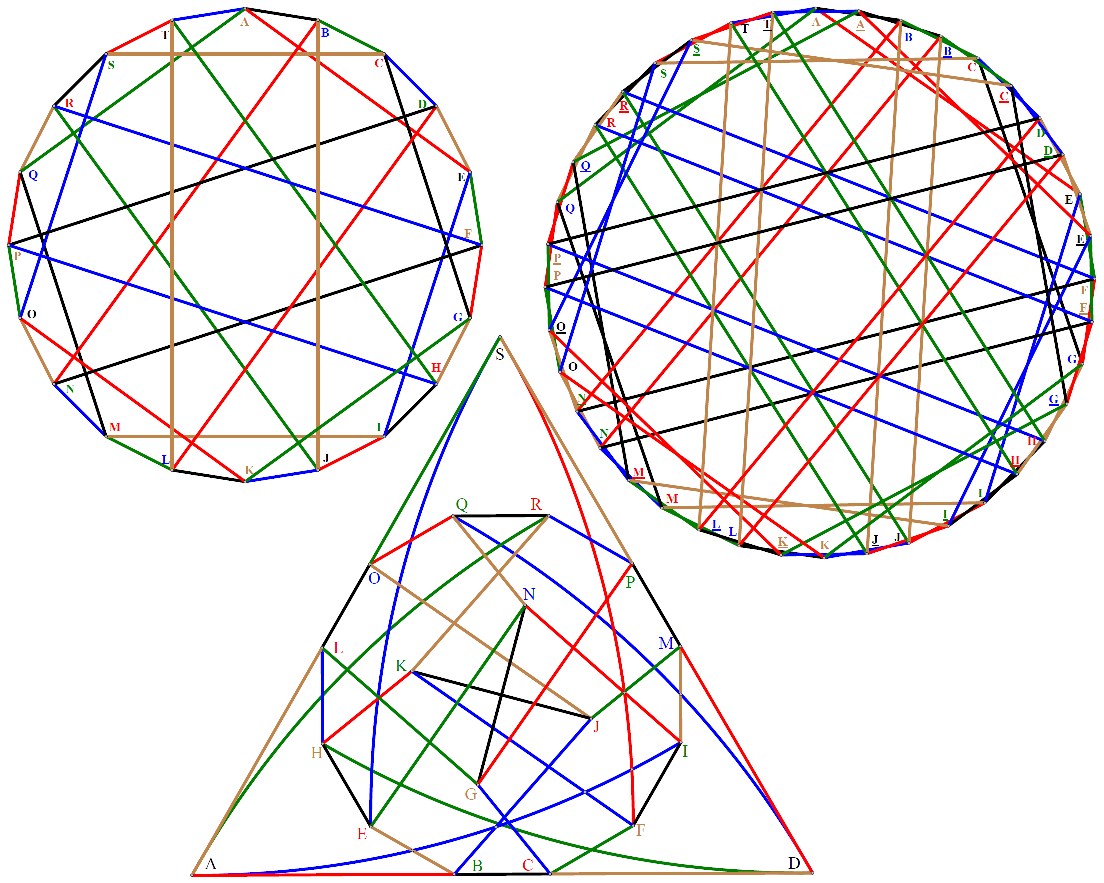}
\caption{TCs of 20-, 40- and 19-vertex 4-regular graphs of girth 5.}
\label{penta20}
\end{figure}

\begin{proof} The graph $\Gamma$ is mentioned in \cite[Section 4]{AFLN}. Figure~\ref{penta20} depicts $\Gamma$ bearing a TC which is not efficient. In the figure, the vertices of $\Gamma$ are represented with the capital letters from A to T, that we number correspondingly as vertices $v_0,v_1,\ldots, v_{19}$. The letters representing these 20 vertices and the edges of $\Gamma$ are colored in the figure with red=0, blue=1, hazel=2, black=3 and green=4, defining a TC of $\Gamma$. We will see that $\Gamma$ contains exactly 54 5-cycles, but the TC in Figure~\ref{penta20} has just 14 5-cycles with their vertex and edge sets in bijective correspondence with the color set $[5]$. The remaining 40 5-cycles do not have such bijective correspondence. 

$\Gamma$ contains four pairwise disjoint special 5-cycles represented in the figure via two regular pentagons (a central one, say $P_1$, with its interior in light-gray, and an external one, say $P_0$) 
and two regular pentagrams (one of them, say $Q_0$, with its isosceles triangles in yellow, and the other one, say $Q_1$, obtained by symmetry from $Q_0$ about the vertical line passing through vertex A=$v_0$). Seen as 5-cycles of $\Gamma$, we have that 

$P_0=($AEIMQ$)=(v_0v_4v_8v_{12}v_{16})$, $P_1=($CGKOS$)=(v_2v_6v_{10}v_{14}v_{18})$, 

$Q_0=($BJRFN$)=(v_1v_9v_{17}v_5v_{13})$, $Q_1=($THPDL$)=(v_{19}v_7v_{15}v_3v_{11})$.

\noindent The four copies  
of $Pet$ are obtained as subgraphs of $\Gamma$ by joining respectively:
\begin{enumerate}
\item $P_0$ and $Q_0$ via 1-factor 
$F_0^0=\{$AB,EF,IJ,MN,QR$\}=\{v_0v_1,v_4v_5,v_8v_9,v_{12}v_{13},v_{16}v_{17}\}$;
\item $P_0$ and $Q_1$ via 1-factor
$F_0^1=\{$AT,ED,IH,ML,QP$\}=\{v_0v_{19},v_4v_3,v_8v_7,v_{12}v_{11},v_{16}v_{15}\}$;
\item $P_1$ and $Q_0$ via 1-factor
$F_1^0=\{$CB,GF,KJ,ON,SR$\}=\{v_2v_1,v_6v_5,v_{10}v_9,v_{14}v_{13},v_{18}v_{17}\}$;
\item $P_1$ and $Q_1$ via 1-factor
$F_1^1=\{$CD,GH,KL,OP,ST$\}=\{v_2v_3,v_6v_7,v_{10}v_{11},v_{14}v_{15},v_{18}v_{19}\}$.
\end{enumerate}

Apart from $P_0,P_1,Q_0,Q_1$, there are other 50 5-cycles in $\Gamma$, yielding a total of 54.
There are three {\it types} of edges of $\Gamma$, namely those belonging exactly to:
\begin{enumerate}
\item  eight 5-cycles, a type formed by the 10 edges in $P_0\cup P_1$;
\item seven 5-cycles, a type formed by the 10 edges in $Q_0\cup Q_1$;
\item five 5-cycles, a type comprising the 20 edges of $\Gamma$ not in $P_0\cup P_1\cup Q_0\cup Q_1$.  
\end{enumerate}
$\Gamma$ is represented in the upper-left of Figure~\ref{penta20}.
The 20-cycle $\Sigma=$(AB$\cdots$ST)$=$ 
$(v_0v_1\cdots v_{18}v_{19})$ containing all the edges of $\Gamma$ not in $P_0\cup P_1\cup Q_0\cup Q_1$, contains exactly 20 paths of length  2 each that alternatively belong exactly to one 5-cycle and to two 5-cycles, respectively in the following two sets of ten paths each:
\begin{eqnarray}\begin{array}{l}\label{almadraba}
\{\mbox{TAB,BCD,DEF,FGH,HIJ,JKL,LMN,NOP,PQR,RST}\}=\{v_{19}v_0v_1,\ldots,v_{17}v_{18}v_{19}\};\\  
\{\mbox{ABC,CDE,EFG,GHI,IJK,KLM,MNO,OPQ,QRS,STA}\}=\{v_0v_1v_2,\ldots,v_{18}v_{19}v_0\}.
\end{array}\end{eqnarray}
We distinguish four  {\it types} of 5-cycles in $\Gamma$, namely:

\begin{enumerate}
\item $P_0$, $P_1$, $Q_0$ and $Q_1$;
\item 10 5-cycles with 1 edge $v_{2i}v_{2i+4}$ in $P_0\cup P_1$ and 1 path $v_{2i}v_{2i+1}v_{2i+2}v_{2i+3}v_{2i+4}$, $i\in\mathbb{Z}_{20}$;
\item 20 5-cycles with 1 path of length 2 in $P_0\cup P_1$, 1 edge in $Q_0\cup Q_1$ and 2 edges in $\Sigma$;
\item 20 5-cycles with 1 path of length 2 in $Q_0\cup Q_1$, 1 edge in $P_0\cup P_1$ and 2 edges in $\Sigma$.
\end{enumerate}f
The 5-cycles of types 1, 2, 3 and 4 are listed in the first, second, third-fourth and fifth-sixth columns in the lower part of Figure~\ref{penta20}, respectively. 
In the representation of $\Gamma$ in Figure~\ref{penta20}, each 5-cycle of types 1 and 2 has its vertex set and its edge set each in bijective correspondence with the color set, which is not the case for the 5-cycles of types 3 and 4; for help, see displays (\ref{ve})-(\ref{pq}) below. In addition, each path of length 2 in display (\ref{almadraba}) belongs to one 5-cycle or two 5-cycles of type 2.

Each edge in:
\begin{enumerate}
\item $P_0\cup P_1$. has its 8 5-cycles being: 1 each of types 1 and 2, 4 of type 3 and 2 of type 4;
\item $Q_0\cup Q_1$ has its 7 5-cycles being: 1 of type 1, 2 of type 3 and 4 of type 4;
\item $\Sigma$ has its 5 5-cycles being: 2 of type 2, 2 of type 3 and 1 of type 4. 
\end{enumerate}

\noindent In the upper-left of Figure~\ref{penta20}, we have the following correspondences from the successive elements of the 20-cycle $\Sigma$ onto the color set $[5]$, where the assignment of the colors appears as an exponent or super-index in the case of each vertex $v_i$ and each edge $e_i=v_iv_{i+1}$, with subindex $i$ mod 20:

\begin{eqnarray}\begin{array}{l}\label{ve}
V(\Sigma)=(v_0^2v_1^1v_2^0v_3^4v_4^3v_5^2v_6^1v_7^0v_8^4v_9^3v_{10}^2v_{11}^1v_{12}^0v_{13}^4v_{14}^3v_{15}^2v_{16}^1v_{17}^0v_{18}^4v_{19}^3),\\
E(\Sigma)=e_0^3e_1^4e_2^1e_3^2e_4^4e_5^0e_6^2e_7^3e_8^0e_9^1e_{10}^3e_{11}^4e_{12}^1e_{13}^2e_{14}^4e_{15}^0e_{16}^2e_{17}^3e_{18}^0e_{19}^1).
\end{array}\end{eqnarray}
Using such notation, we also have the coloring restricted to $P_0\cup P_1\cup Q_0\cup Q_1$, by inserting the corresponding colors between each two adjacent vertices:
\begin{eqnarray}\begin{array}{l}\label{pq}
P_0=(v_0^20v_4^31v_8^42v_{12}^03v_{16}^14),\;
P_1=(v_2^03v_6^14v_{10}^20v_{14}^31v_{18}^42),\\
Q_0=(v_1^12v_9^34v_{17}^01v_5^23v_{13}^40),\;Q_1=(v_3^40v_{11}^12v_{19}^34v_7^01v_{15}^23).\\
\end{array}\end{eqnarray}
\end{proof}

\begin{proposition}
A double cover $\Lambda$ of the graph $\Gamma$ of Theorem~\ref{Gamma} exists in which the four copies of $Pet$ in $\Gamma$ are covered by corresponding copies of $Dod$ in $\Lambda$. In fact, $\Lambda$ is a 40-vertex 4-regular graph of girth 5 and 64 5-cycles. Moreover, $\Lambda$ has a total non-efficient coloring containing just 24 5-cycles with their vertex and edge sets in bijective correspondence with the color set $[5]$. The remaining 40 5-cycles do not have such bijective correspondence.
\end{proposition}

\begin{proof}
The upper-right Figure~\ref{penta20} contains a representation of $\Lambda$ inheriting the induced colorings of the TC of $\Gamma$, where each vertex X in $\Gamma$ is represented by two vertices $X$ and $\ul{X}$ in $\Lambda$, with X$\in\{$A, $\ldots, $Z$\}$. The TC of $\Lambda$ is well defined and compatible with the inherited colorings.
Moreover, $\Lambda$ is the edge-disjoint union of the following cycles, two of length 20 (threading through the four copies of $Pet$) and two of length 10 (twice shown in the copies):
\begin{eqnarray}\begin{array}{l}\label{eqn1}
(A,B,C,D,E,F,G,H,I,J,K,L,M,N,O,P,Q,R,S,T)=(v_0,v_1,\ldots,v_{19});\\
(\ul{A},\ul{B},\ul{C},\ul{D},\ul{E},\ul{F},\ul{G},\ul{H},\ul{I},\ul{J},\ul{K},\ul{L},\ul{M},\ul{N},\ul{O},\ul{P},\ul{Q},\ul{R},\ul{S},\ul{T})=(w_0,w_1,\ldots,w_{19});\\
(\ul{B},N,\ul{F},R,\ul{J},B,\ul{N},F,\ul{R},J)=(w_1,v_{13},w_5,v_{17},w_9,v_1,w_{13},v_5,w_{17},v_9);\\
(\ul{T},L,\ul{D},P,\ul{H},T,\ul{L},D,\ul{P},H)=(w_{19},v_{11}w_3,v_{15},w_7,v_{19},w_{11},v_3,w_{15},v_7).\end{array}\end{eqnarray}
Each 5-cycle in the set of 32 5-cycles formed by $P_0,P_1$ and all those in the second, fifth and sixth columns in Figure~\ref{penta20} lifts into two 5-cycles of $\Lambda$, yielding a total of 64 5-cycles of $\Lambda$, e.g. $(A,E,I,M,Q)=(v_0,v_4,v_8,v_{12},v_{16})$, $(\ul{A},\ul{E},\ul{IM},\ul{Q})=(w_0w_4w_8w_{12}w_{16})$, etc. 
The remaining 22 5-cycles in the columns in Figure~\ref{penta20} are lifted onto 22 corresponding 10-cycles of $\Lambda$, including the 10-cycles in display (\ref{eqn1}).  
\end{proof}

\begin{theorem}
The Robertson 19-vertex $(4,5)$ cage~\cite{Rob} admits a TC which is nonefficient.
\end{theorem}

\begin{proof} Let $\Gamma$ be the Robertson graph.
The claimed TC of $\Gamma$ is represented to the bottom of Figure~\ref{penta20}, where $V(\Gamma)=\{A,B,\ldots, S\}=\{v_0,v_1,\ldots,v_{18}\}$, with colors red=0, blue=1, hazel=2, black=3 and green =4.
Notice that $\Gamma$ has 44 5-cycles, but only four of them have both their vertex sets and edge sets in bijection with $[5]$ under the TC, namely $(DMIFC)$, $(AINGL)$, $(SFKJO)$ and $(SENGP)$. Thus, the shown TC of $\Gamma$ is nonefficient. In addition, only five 5-cycles has their edge sets in bijection with $[5]$, but not their vertex sets, namely $(ABEHL)$, $(SEBCF)$, $(ABCGL)$, $(DCBJM)$ and $(ARKJB)$. The remaining 35 4-cycles do not have neither their edge sets not their vertex sets in bijection with $[5]$.
\end{proof}

\begin{theorem}\label{kaput}
Let $k\ge 2$ be not divisible by 5. Then, the graphs $Pet^k$ and $Dod^k$ ($k\ge 2$) admit TCs which are nonefficient.
\end{theorem}

\begin{proof}
Let $k\ge 2$ be not divisible by 5. Denote the Hamilton $10k$-cycle of $Pet^k$ to which the additional $k$ pentagon and $k$ pentagram 5-cycles of $Pet^k$ are joined alternatively as
$$C=(0,1,2,\ldots,10k-2,10k-1).$$ Such pentagons are expressible as $C_j=(0,2k,4k,6k,8k)+2j$ and such pentagrams are expressible as $C'j=(0,4k,8k,2k,6k)+2j+1$, for $0\le j<k$, where addition of $2j$ of $2j+1$ affects all the terms of the 5-cycles. To get the claimed nonefficient TC of $Pet^k$, 
we define the set of vertices and edges having color 0 as follows:
$$\bigcup_{j=0}^{k-1}(\{0,(2k,8k),(4k,6k),5,(3,4),(6,7)\}+10j),$$ where the operation of adding $10j$, ($0\le j<k$), means adding $10j$ to each vertex or edge endvertex.
Notice that this is not well defined if $k$ is divisible by 5, for in such a case each pentagon or pentagram is monochromatic.
To obtain the set of vertices and edges having color $i$, ($0\le i<k$), we just add $i$ to each vertex of edge endvertex of color 0. This yields a nonefficient TC of  
$Pet^k$. Then, it is elementary to extend this to a nonefficient TC of $Dod^k$, for which we recall that $V(Dod)$ is formed by: {\bf(a)} two $10k$ cycles $C$ and $C'$, where $C$ is given above and $C'=0',1',2',\ldots,(10k-1)'$; {\bf(b)} the $2k$ 5-cycles $C_j$ and $C'_j$ given above plus the $2k$ new 5-cycles $C''_j=(0',(2k)',(4k)',(6k)',(8k)')+2j$ and $C'''j=(0',(4k)',(8k)',(2k)',(6k)')+2j+1$. In addition, for each edge $(x,y)$ of $Pet^k$, we have two edges of $Dod^k$, namely $(x,y')$ and $(x',y)$.
\end{proof}

\begin{example}
Figure~\ref{caputti} sketches in gray the graphs $Pet^k$, for $k=2,3,4,6$, except for those vertices and edges having color 0, which are indicated in red or with a red 0.
\begin{figure}[htp]
\includegraphics[scale=0.88]{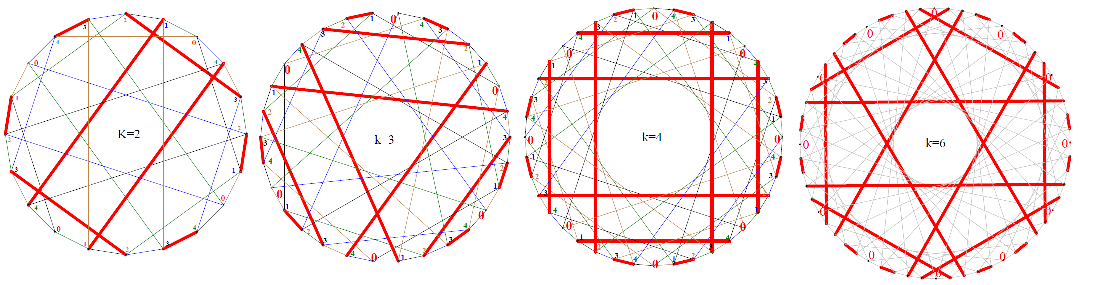}
\caption{Stressing in red those vertices and edges in color 0 in $Pet^k$, for $k=2,3,4,6$.}
\label{caputti}
\end{figure}
\end{example}

\begin{corollary}
For $k>2$ not divisible by 5, the $2k$ 5-cycles formed by the $k$ pentagons and $k$ pentagrams of $Pet^k$ are the only 5-cycles of $Pet^k$. They form a set of $2k$ disjoint 5-cycles and their union $U$ is the complement of the Hamilton cycle $C$ of $Pet^k$. The restriction of the TC in Theorem~\ref{kaput} to each of these 5-cycles establishes bijections between their vertex sets and edge sets with the color set $[5]$. A similar result holds for $Dod^k$, where $U$ is enlarged to the union $U'$ of the $4k$ 5-cycles of $Dod^k$. In each such 5-cycle each color is assigned just to a specific vertex and to its opposite edge.
\end{corollary}

\begin{proof}
The higher value of $k>2$ allows only cycle lengths larger than 5 for those cycles which are not the $k$ pentagons and $k$ pentagrams attached alternatively to the vertices of the Hamilton $10k$-cycle $C$.
\end{proof}

\end{document}